\documentclass[11pt, a4paper]{article}
\usepackage[latin1]{inputenc}
\usepackage[T1]{fontenc}
\usepackage[english]{babel}
\usepackage{amssymb}
\usepackage{amsmath}
\usepackage{amsthm}
\usepackage[colorlinks=true, allcolors=blue]{hyperref}
\usepackage{geometry}
\usepackage{fancyhdr}
\usepackage{relsize}
\usepackage{upgreek}
\usepackage{verbatim}
\usepackage{enumitem}
\usepackage{calrsfs}
\usepackage{tikz}
\usetikzlibrary{matrix,arrows,calc}
\usepackage{mypackage}
\title{Chow Rings of Fine Quiver Moduli are Tautologically Presented}
\author{H. Franzen\footnote{\myaddress}}
\date{}
\begin{document}
	\maketitle
	\begin{abstract}
	A result of A. King and C. Walter asserts that the Chow ring of a fine quiver moduli space is generated by the Chern classes of universal bundles if the quiver is acyclic. We will show that defining relations between these Chern classes arise geometrically as degeneracy loci associated to the universal representation.
\end{abstract}

	\section*{Introduction\markboth{}{Introduction}}
Many interesting moduli spaces are varieties parametrizing stable objects in an abelian category up to isomorphism \cite{Newstead:78}. In this paper, we study moduli spaces of quiver representations \cite{King:94}. These varieties can be regarded as an organizing tool for the representation theory of a given quiver. One would like to gather as much information as possible about the structure of quiver moduli and relate it to the representation theory of the quiver.\\
Just like the classification problem for representations of quivers is a hard problem in general, their moduli spaces are very difficult as well. Therefore, we concentrate on describing their invariants and relating these to representation theory. In case we are dealing with a fine moduli space of an acyclic quiver, where the moduli space is a non-singular projective variety, the Chow ring is an invariant with many nice properties \cite{Fulton:98}. For (fine) quiver moduli, we know that the Chow ring is isomorphic to the cohomology ring, thanks to \cite[Theorem 3]{KW:95}. We will give a presentation of this ring relying on the representation theory of the quiver.

So far, it is known that, if the quiver is acyclic, Chow rings of fine quiver moduli are tautologically generated, which means that they are generated by Chern classes of the universal representation \cite[Theorem 3]{KW:95}. In this context, a moduli space is fine if such a universal representation exists. Furthermore, there are ways to compute the Betti numbers of quiver moduli, i.e. the dimensions of the graded pieces of the Chow ring, e.g. by resolved Harder-Narasimhan recursion \cite[Theorem 6.7]{Reineke:03} or by the MPS formula \cite[Corollary 3.7]{RSW:12}.\\
It is quite natural to ask for a presentation of these rings and whether such a presentation arises as tautologically as the generators do. In fact, it is reasonable to assume that this is true: For the Grassmannian of $n$-dimensional subspaces in an $r$-dimensional vector space, which can be regarded as a moduli space for the $r$-arrow Kronecker quiver (with a suitable dimension vector), the Chow ring is just made that way. It is generated by the Chern classes of the universal rank $(n-r)$-bundle, which obey \emph{only} the obvious relations coming from the fact that it is a quotient bundle of the trivial bundle of rank $r$ \cite[Theorem 1]{Groth:58:inter}.\\
Relying on the construction of the moduli space as a Geometric Invariant Theory quotient \cite{GIT:94}, and restricting to a certain stability condition, we might apply a celebrated result \cite[Theorem 4.4]{ES} which gives a presentation for the Chow ring of a geometric quotient. However, this result turns out to be not well-adapted to the representation-theoretic context of quiver moduli: As generators and relations in this result come from choosing certain maximal tori and relating the toric geometry to the original quotient, we lose track of the representation theory of the quiver.\\
Regarding the moduli space as a symplectic reduction \cite[Corollary 6.2]{King:94}, we can use results of \cite{Kirwan:84} to give a description of its cohomology ring. For fine quiver moduli, this is the same because the cycle map is an isomorphism. Yet, calculating the Kirwan ideal of a symplectic reduction seems to be a difficult task. Even a harmless looking moduli space, like $m$ ordered points in $\P^1$ up to $\PGl_2$-action, requires substantial effort and elaborate methods to obtain a presentation of the cohomology ring via the Kirwan ideal \cite[Theorem 5.5]{HK:98}.

The approach to obtaining relations between Chern classes of universal bundles, i.e. bundles occurring in the universal representation, is to try to reformulate the notion of stability in terms of the universal representation. The universal representation allows us to keep track of some representation-theoretic facts if we pass to the level of Chow rings.\\
The (slope-)stability condition of the moduli space tells us which sub-dimension vectors of the given dimension vector are allowed and which ones are forbidden. We will find for every forbidden sub-dimension vector one forbidden polynomial in the Chern \emph{roots} of the universal bundles. Every forbidden polynomial gives a set of polynomials in the Chern classes of universal bundles that we call the \emph{tautological relations}. The main result of this paper, Theorem \ref{thm}, states that the tautological relations give - up to one extra linear relation, arising for trivial reasons - a presentation of the Chow ring. As this linear relation is fairly easy to handle, we should not worry too much about it.\\
As an application, we obtain a presentation for the moduli space of $m$ ordered points in $\P^1$ up to $\PGl_2$-action that resembles the one obtained in \cite{HK:98}. Choosing a different representative of the universal representation, we are also able to give a presentation that preserves the natural action of the symmetric group on the Chow ring that comes from permuting the points.\\
Applying the result to the $r$-arrow Kronecker quiver (with dimension vector $(1,r-n)$) gives a defining set of relations for the Chow ring of the Grassmannian of $n$-dimensional subspaces in an $r$-dimensional vector space that can easily be transformed into the well-known presentation of \cite[Theorem 1]{Groth:58:inter}. We also give an example of a 6-dimensional variety, a so-called Kronecker module (cf. \cite[III-VI]{Drezet:88} or \cite[Section 6]{ES}), for which an explicit presentation of the Chow ring was not previously known.

To find the tautological relations, we have to pass to a complete flag bundle where (the pull-back of) every universal bundle splits. The reason for this is that, in general - unlike in the case of the Grassmannian, for example - a stability requirement cannot be formulated directly in terms of the universal representation working only over the moduli space itself. Yet, in the flag bundle, we can find relations in the Chern roots of the universal bundles (cf. Corollary \ref{ko1}). Expressing these relations in terms of a basis of the Chow ring of the flag bundle over its base yields the tautological relations.\\
The proof of the result that the Chow ring is, in fact, presented by the tautological relations (and the linear relation) proceeds in two steps. It is similar to the proof of the main result in \cite{ES}. First, we show that this is true in the toric case, i.e. the case where the dimension vector consists of ones entirely. Therefore, it is essential to obtain a detailed description of the toric fan of the moduli space (Lemma \ref{simplex} and Proposition \ref{toric_fans}, cf. also \cite{Hille:98}). As a second step, we reduce the arbitrary case to the toric one by using a covering quiver of the original quiver. We regard the moduli space as a geometric quotient by some reductive algebraic group. Hence, passing to the covering quiver amounts to choosing a maximal torus of the group. The crucial part is to control the difference between stability with respect to the group and stability for the torus, when passing to the Chow ring. We do this entirely algebraically by applying Poincar\'{e} duality and some 
knowledge about the action of the Weyl group on the Chow ring (Lemmas \ref{anti_invar} 
and \ref{degrees}). Finally, we prove that the ring which is claimed to be isomorphic to the Chow ring also fulfills Poincar\'{e} duality (Lemma \ref{perfPair}). This forces the desired isomorphism.

The paper is organized as follows: In Section 1, we recall the notions and facts on quiver representations and their moduli spaces. We explain where tautological relations come from in Section 2, whereas Section 3 is devoted to the proof of the main result Theorem \ref{thm}. Afterwards, we give applications of this very result in Section 4.

\begin{ack*}
	While doing this research, I was supported by the priority program SPP 1388 ``Representation Theory'' of the DFG (German Research Foundation). I am grateful to M. Reineke for his patient support and for very inspiring discussions. I would also like to thank C. Chindris for pointing out that Proposition \ref{thm_toric} holds true over a ground field of arbitrary characteristic and with integral coefficients.
\end{ack*}
	\section{Quiver representations and their moduli spaces}
\begin{conv*}
	Fix an algebraically closed field $\kk$ of characteristic $\cha \kk = 0$. All vector spaces will be $\kk$-vector spaces and all varieties will be $\kk$-varieties. When we talk about a point $x$ of a variety $X$, we mean a $\kk$-valued point of $X$.
\end{conv*}
Let $Q$ be a quiver. In our context, quivers are assumed to have a finite set of vertices $Q_0$ and a finite set of arrows $Q_1$. Most of the time, we will also assume $Q$ to be connected. Let $X$ be a variety. A \textbf{representation} $M$ of $Q$ over $X$ consists of vector bundles $M_i$ on $X$ for every $i \in Q_0$ and homomorphisms $M_\alpha: M_i \to M_j$ of vector bundles on $X$ for every arrow $\alpha: i \to j$ of $Q$. If we speak about a representation of $Q$ without mentioning the variety it lives over, we mean a representation over the variety $\Spec \kk$. Thus, this representation consists of finitely generated vector spaces and linear maps between them.\\
A homomorphism of representations of $Q$ over $X$ is defined in the obvious way, so we obtain the category $\cat{Rep}_X(Q)$ of representations of $Q$ over $X$ (cf. \cite[Chapter II]{ASS:06} for details on representations of quivers). Furthermore, for any morphism $f: Y \to X$ of varieties and any representation $M$ of $Q$ over $X$, we can form the pull-back via $f$ of every $M_i$ and get a representation $f^*M$ of $Q$ over $Y$. This pull-back behaves functorially (see \cite{GK:05} for more). If $x$ is a point of $X$ and $j: \Spec \kk \to X$ denotes the corresponding closed embedding, the fiber $M_x := j^*M$ is a representation of $Q$.\\
Let $M$ be a representation of $Q$ over a variety $X$ and denote by $d_i$ the rank of the vector bundle $M_i$. The family $d = (d_i \mid i \in Q_0)$ is called the \textbf{dimension vector} of $M$.

We fix a quiver setting $(Q,d)$, that is a pair consisting of a quiver $Q$ and a dimension vector $d = (d_i \mid i \in Q_0)$. We are interested in a variety that parametrizes isomorphism classes of representations of $Q$ with dimension vector $d$ in a natural way. That is why we consider the contravariant functor $\Rep(Q,d)$ from the category of varieties to the category of sets that associates to every variety $X$ the set $\Rep_X(Q,d)$ of equivalence classes of representations of $Q$ over $X$ with dimension vector $d$ (for short: a representation of $(Q,d)$ over $X$). We call representations of $Q$ over $X$ \textbf{equivalent} if all their fibers are isomorphic. This functor is usually not representable, we have to pass to (equivalence classes of) (semi-)stable representations.

\begin{defn*}
	A \textbf{stability condition} for a quiver $Q$ is a $\Z$-linear map $\theta: \Z^{Q_0} \to \Z$.
\end{defn*}

Fix a stability condition $\theta$ for $Q$ and define the corresponding \textbf{slope} $\mu := \mu_\theta: \Z_{\geq 0}^{Q_0} - \{0\} \to \Q$ by
$$
	\mu(d) := -\frac{\theta(d)}{\dim d}
$$
where $\dim d = \sum_i d_i$. For shortness, write $\mu(M) := \mu(d)$ if $M$ is a representation of $Q$ of dimension vector $d$. We will define (semi-)stability by means of this slope function. This has the advantage of being able to define (semi-)stability of a representation of $Q$ regardless of its dimension vector.

\begin{defn*}
	Let $M$ be a representation of $Q$ over $X$. We call $M$ $\theta$-\textbf{semi-stable} if $\mu(M') \leq \mu(M)$ holds for all non-zero sub-representations $M'$ of $M$. Otherwise, it will be called $\theta$-\textbf{unstable}. A $\theta$-semi-stable representation $M$ (of $Q$ over $X$) is called $\theta$-\textbf{stable} if the only sub-representation $M' \neq 0$ that satisfies $\mu(M') = \mu(M)$ is $M$ itself.
\end{defn*}

\begin{rem*}
	If $\theta(d) = 0$ then a representation is semi-stable if and only if $\mu(M') \leq 0$ for every sub-representation $M'$ which, in turn, is satisfied precisely if $\theta(M') \geq 0$ (note the sign in the definition of $\mu$). The reason we make this sign convention is that semi-stability with respect to a slope function usually means ``no sub-object has bigger slope'' while King's original definition of semi-stability of a representation with respect to a stability condition (cf. \cite[Definition 1.1]{King:94}) reads ``every sub-representation $M'$ satisfies $\theta(M') \geq 0$''.
\end{rem*}

Let $d$ be a dimension vector for $Q$. The sets of equivalence classes of all $\theta$-semi-stable/$\theta$-stable representations of $(Q,d)$ over $X$ will be denoted $\Rep_X^{\theta-\sst}(Q,d)$ and $\Rep_X^{\theta-\st}(Q,d)$, respectively. Note that the notions of $\theta$-semi-stability and $\theta$-stability coincide if the dimension vector $d$ is $\theta$-\textbf{coprime}, that means if $\mu(d') \neq \mu(d)$ for all sub-dimension vectors $0 \leq d' \leq d$ with $0 \neq d' \neq d$. In this case, we simply write $\Rep^\theta_X(Q,d)$. In general, the $\theta$-stable representations of $Q$ with slope $\mu_0$ are precisely the simple objects in the full abelian subcatetory of $\cat{Rep}_X(Q)$ of $\theta$-semi-stable representations of $Q$ with a fixed slope $\mu_0$.\\
We call $d$ \textbf{coprime} if the greatest common divisor of all $d_i$ is one. Let us also remark that if $d$ is coprime, $d$ will also be $\theta$-coprime for a generic choice of $\theta$ (that means $\theta$ avoids finitely many hyperplanes in the space of stabilities).

\begin{conv*}
	In what follows, let $(Q,d)$ be a quiver setting with a \emph{connected} quiver $Q$ and a \emph{coprime and $\theta$-coprime} dimension vector $d$.
\end{conv*}

Then, we can in fact construct a representing object of the contravariant functor $\Rep^{\theta}(Q,d)$ that associates to every variety $X$ the set $\Rep_X^{\theta}(Q,d)$. This is an open sub-functor of $\Rep(Q,d)$ mentioned above.\\
We sketch this construction briefly: For every $i \in Q_0$, fix a vector space $V_i$ of dimension $d_i$. For every arrow $\alpha: i \to j$ of $Q$, let $R_\alpha := \Hom(V_i,V_j)$, and we define
$$
	R := R(Q,d) := \bigoplus_{\alpha \in Q_1} R_\alpha.
$$
We consider elements $M = \sum_\alpha M_\alpha \in R$ as representations of $Q$ on the fixed vector spaces $V_i$. As $R$ is a vector space, it possesses the structure of an affine space. Now, define an action of the reductive linear algebraic group $G := G(Q,d) := \prod_{i \in Q_0} \Gl(V_i)$ on $R$ via
$$
	g \cdot M = \sum_{\alpha:i \to j} g_jM_\alpha g_i^{-1},
$$
where $g = (g_i \mid i \in Q_0) \in G$ and $M = \sum_\alpha M_\alpha \in R$. The image $\Gamma$ of the diagonal embedding $\G_m \to G$ acts trivially, whence we obtain an action of the quotient group $PG := PG(Q,d) := G/\Gamma$ (which is also a reductive linear algebraic group). Two representations $M, M' \in R$ are isomorphic if and only if they lie in the same $PG$-orbit. Thus, set-theoretically, the set of orbits is the right thing to consider. The hard part is to equip this with the structure of a variety and a representation of $(Q,d)$ over this variety. This fails in general. However, on the subset $R^{\theta} := R^{\theta}(Q,d)$ of $\theta$-(semi-)stable elements of $R$, Mumford's geometric invariant theory (cf. \cite{GIT:94}) asserts that a geometric quotient exists. In our context, it reads as follows, as it has been pointed out in \cite[Proposition 3.1]{King:94}:

\begin{thm}[Mumford, King]
	There exists a geometric $PG$-quotient $\pi: R^{\theta} \to M^{\theta}(Q,d)$.
\end{thm}

Note that for a $\theta$-stable representation $M$, Schur's Lemma holds, that means the only automorphisms of $M$ are the scalars (here, we need $Q$ to be connected). Thus $PG$ acts freely on $R^\theta$ which is a necessary condition for the geometric quotient to exist. Most of the time, we denote this quotient just $M^\theta$. One can show that $M^\theta$ has some good properties (cf. \cite[Proposition 4.3, Remark 5.4]{King:94}):

\begin{thm}[King]
	Consider the geometric quotient $\pi: R^\theta \to M^{\theta}$.
	\begin{enumerate}
		\item The variety $M^{\theta}$ is non-singular.
		\item If the quiver $Q$ is acyclic, $M^{\theta}$ is also a projective variety.
	\end{enumerate}
\end{thm}

The final ingredient is a representation $U$ of $(Q,d)$ over $M^{\theta}$ with the obvious universal property making $U$ a representing object of the functor ``(semi-)stable representations of $(Q,d)$ up to isomorphism''. For every $i \in Q_0$, let $E_i = V_i \times R^{\theta}$ be the trivial vector bundle on $R^{\theta}$. Define an action of $G$ on $E_i$ as follows: As the dimension vector $d$ is coprime, there exist integers $a_i$ with $\sum_i a_id_i = 1$. Define a character $\psi: G \to \G_m$ of weight one by $\psi(g) = \prod_i(\det g_i)^{a_i}$. This choice assures that with
$$
	g \cdot (v,M) = (\psi(g)^{-1}g_iv, g\cdot M)
$$
for $g \in G$, $v \in V_i$, and $M \in R^{\theta}$, the image $\Gamma \sub G$ of the diagonal embedding $\G_m \to G$ acts trivially on every fiber of $E_i$. Therefore, we obtain a $PG$-action on $E_i$, whence the bundle $E_i$ descends to a (uniquely determined) vector bundle $U_i$ on $M^{\theta}$, i.e. $\pi^*U_i = E_i$. This works for the maps as well: Let $\alpha: i \to j$ be an arrow of $Q$. Then define the map $E_\alpha: E_i \to E_j$ by
$$
	E_\alpha(v,M) = (M_\alpha v,M).
$$
This map is $PG$-equivariant, thus it also descends to a homomorphism $U_\alpha: U_i \to U_j$ of vector bundles on $M^{\theta}$. Denote those two representations over $R^{\theta}$ and $M^{\theta}$ with $E$ and $U$, respectively. This representation $U$ can be seen as a universal object (cf. \cite[Proposition 5.3]{King:94}):

\begin{thm}[King]\label{king1}
	The representation $U$ of $(Q,d)$ over the geometric quotient $M^{\theta} = M^{\theta}(Q,d)$ is a representing object of the contravariant functor $\Rep^{\theta}(Q,d)$.
\end{thm}

The representation $U$ is called a \textbf{universal representation} of $(Q,d)$. Note that a universal representation is not unique up to isomorphism, but just up to equivalence, which is a much weaker condition as it is defined pointwise on $M^\theta$. In fact, different choices of a character $\psi$ give rise to non-isomorphic universal representations (note that twisting a universal representation with a line bundle gives again a universal representation which is not isomorphic to the original one unless the line bundle is trivial).
	\section{Tautological relations in the Chow ring of $M^{\theta}(Q,d)$}
Our goal is to give an explicit description of the Chow ring of a fine quiver moduli. Before doing so, let us briefly recollect some facts on Chow rings.

\subsection{A short reminder on intersection theory}

We assume the reader is familiar with the basic notions of intersection theory. As a main reference, we recommend Fulton's book \cite{Fulton:98}.

If $X$ is a variety, or more generally an algebraic scheme over $\kk$, we denote by $A_*(X) = \bigoplus_{k \geq 0} A_k(X)$ the Chow group of $X$ (cf. \cite[1.3]{Fulton:98}). Remember that there exists a proper push-forward and a flat pull-back (cf. \cite[1.4 and 1.7]{Fulton:98}) as well as Gysin pull-backs for lci morphisms (cf. \cite[Chapter 6]{Fulton:98}).

In case $X$ is non-singular of dimension $n$, we write $A^i(X) = A_{n-i}(X)$. There exists an intersection product making $A^*(X) = \bigoplus_{i \geq 0} A^i(X)$ a commutative graded ring (cf. \cite[Chapter 8]{Fulton:98}). For a morphism $f: Y \to X$ of non-singular varieties (which is necessarily lci), the Gysin pull-back gives a homomorphism $f^*: A^*(X) \to A^*(Y)$ of graded rings (cf. \cite[8.3]{Fulton:98}). 

Let $E$ be a vector bundle of rank $r$ on a non-singular variety $X$. Denote by $c_i(E) \in A^i(X)$ the $i$-th Chern class of $E$ (see \cite[3.2]{Fulton:98} or \cite{Groth:58:Chern} for an axiomatic definition of Chern classes). Let $c_t(E)$ be the Chern polynomial in the indeterminate $t$.

In \cite{Groth:58:inter}, Grothendieck has shown that the Chow ring of a flag bundle is closely related to the Chow ring of its basis: Let $E$ be a vector bundle of rank $r$ on $X$. We still assume $X$ to be non-singular. Consider the complete flag bundle $\Fl(E) \to X$ of $E$. Let $U^\bull$ be the universal complete flag of the pull-back $E_{\Fl(E)}$ and let $\xi_i := c_1(U^i/U^{i-1})$. Then, \cite[Theorem 1]{Groth:58:inter} states:

\begin{thm}[Grothendieck] \label{groth_orig}
	As a graded $A^*(X)$-algebra, $A^*(\Fl(E))$ is generated by $\xi_1,\ldots,\xi_r$, subject only to those relations contained in the expression
	$$
		c_t(E) = \prod_{i=1}^r (1 + \xi_i t).
	$$
\end{thm}

Let $E$ be a vector bundle of rank $r$ on a non-singular variety $X$ of dimension $n$. Let $s: X \to E$ be a section of $E$ and consider its set $Z(s)$ of zeros. There exists a class $\Z(s) \in A_{n-r}(Z(s))$, called the \textbf{localized top Chern class} of $s$, which has the following properties (cf. \cite[14.1]{Fulton:98}):
\begin{enumerate}
	\item $i_*\Z(s) = c_r(E)$, where $i: Z(s) \to X$ is the closed embedding,
	\item if $s$ is a regular section, then $\Z(s) = [Z(s)]$, the associated cycle to $Z(s)$ regarded as the (scheme-theoretic) zero fiber of $s$, and
	\item the formation of localized top Chern classes commutes with Gysin maps and proper push-forwards.
\end{enumerate}
The case that we are interested in is the following: Let $f: E \to F$ be a map of vector bundles on $X$ of ranks $r$ and $s$, respectively. This map can be regarded as a global section of the Hom-bundle $\usc{\Hom}(E,F) = E^\vee \otimes F$. Then $Z(f)$ is the set of all $x \in X$ such that the map $f_x: E_x \to F_x$ is the zero map and
$$
	i_*\Z(f) = c_{rs}(E^\vee \otimes F) = \prod_{i=1}^r \prod_{j=1}^s (\eta_j - \xi_i),
$$
where $\xi_i$ and $\eta_j$ are the Chern roots of $E$ and $F$, respectively. Although a little misleading, because the dependency on $f$ vanishes, we will write $\Z(f)$ as an abbreviation for $i_*\Z(f)$, in the following.

\subsection{The Chow ring of $M^\theta(Q,d)$ is tautologically generated}

Let $Q$ be an acyclic quiver and let $d$ be a coprime and $\theta$-coprime dimension vector for $Q$. Fix vector spaces $V_i$ of dimension $d_i$ and let $R = R(Q,d)$, $R^\theta = R^{\theta}(Q,d)$, and $M^\theta = M^{\theta}(Q,d)$ as described in the previous subsection. Our goal is to give a description of the Chow ring $A(M^\theta) := A^*(M^\theta)_\Q$ in terms of generators and relations. Let $U$ be the universal representation of $(Q,d)$ deduced from the character $\psi$ of weight one. We already know the following (cf. \cite[Theorem 3]{KW:95}):

\begin{thm}[King-Walter]\label{king2}
	The Chow ring $A(M^\theta)$ is generated by the Chern classes $c_{i,\nu} := c_\nu(U_i)$, with $i \in Q_0$ and $1 \leq \nu \leq d_i$, as a $\Q$-algebra.
\end{thm}

So, generators for the ring in question are known. We want to gather some information about relations between these generators.

\subsection{Relations between the tautological generators}

There is one rather non-canonical relation which comes from the choice of the character $\psi$. Considering the construction of $U$, it is easy to see that
$$
	\bigotimes_{i \in Q_0} (\det U_i)^{\otimes a_i} = \OO,
$$
where $a_i$ are the integers with $\psi(g) = \prod_i (\det g_i)^{a_i}$. This, in turn, yields the so-called \textbf{linear relation}
$$
	\sum_{i \in Q_0} a_i c_{i,1} = 0.
$$

Next, we will construct certain degeneracy loci in an iterated flag bundle over $M^\theta$ that provide relations between the $c_\nu(U_i)$.

Let us have a closer look at the notion of (semi-)stability. Let $M \in R$ and let $d' \neq 0$ be a dimension vector of $Q$ with $d' \leq d$ and $\mu(d') > \mu(d)$. For convenience, such a sub-dimension vector $d' \leq d$ will be called \textbf{forbidden} for $\theta$ or $\theta$-\textbf{forbidden}. We assume $M$ had a sub-representation $M'$ of dimension vector $d'$. That means there exists a tuple $(V'_i \mid i \in Q_0)$ of subspaces $V_i'$ of $V_i$ of dimension $d'_i$ such that $M_\alpha V_i' \sub V_j'$ holds for every $\alpha: i \to j$. Conversely, $M$ has no sub-representation of dimension vector $d'$ if and only if for all such tuples $(V'_i \mid i \in Q_0)$, there exists an arrow $\alpha: i \to j$ in $Q$ such that
$$
	M_\alpha V'_i \nsubseteq V'_j.
$$

Next, we form for every $i \in Q_0$ the complete flag bundle $\Fl(U_i)$ and denote by $\Fl(U)$ the fiber product of all $\Fl(U_i)$ over $M^\theta$. Let $p: \Fl(U) \to M^\theta$ be the projection. This variety possesses a ``universal flag'' $U^\bull = (U_i^\bull \mid i \in Q_0)$ consisting of complete flags $U_i^\bull$ of $p^*U_i$. The flag $U_i^\bull$ arises as the pull-back of the universal flag on $\Fl(U_i)$ along the natural morphism $\Fl(U) \to \Fl(U_i)$. Together with $U^\bull$, the variety $\Fl(U)$ is the universal $M^\theta$-variety being equipped with a family of complete flags of (the pull-backs of) all $U_i$'s. Therefore, a point $x \in \Fl(U)$ can be regarded as a pair $x = ([M],W^\bull)$ consisting of an isomorphism class of a $\theta$-(semi-)stable representation $M$ of $(Q,d)$ and a tuple $W^\bull = (W_i^\bull \mid i \in Q_0)$ of complete flags $W_i^\bull$ of subspaces of $(U_i)_{[M]} = M_i = V_i$.

We will see that every forbidden sub-dimension vector $d' \leq d$ induces a relation in $A(\Fl(U))$. Fix a forbidden $d' \leq d$. Let $x = ([M], W^\bull)$ be a point of $\Fl(U)$. In particular, we are given a $d'_i$-dimensional subspace of $V_i$ for every $i \in Q_0$. As $M$ is a $\theta$-(semi-)stable representation, we know that there exists an arrow $\alpha:i \to j$ with 
$$
	M_\alpha W_i^{d'_i} \nsubseteq W_j^{d'_j}
$$
or equivalently, the composition of linear maps
$
	W_i^{d'_i} \to V_i \xto{}{M_\alpha} V_j \to V_j/W_j^{d'_j}
$
is not identically zero. We consider the map
$$
	\phi_\alpha^{d'}: U_i^{d'_i} \to p^*U_i \xto{}{p^*U_\alpha} p^*U_j \to p^*U_j/U_j^{d'_j}
$$
of vector bundles on $\Fl(U)$. On the fibers of the point $x \in \Fl(U)$, we obtain
\begin{center}
	\begin{tikzpicture}[description/.style={fill=white,inner sep=2pt}]
		\matrix(m)[matrix of math nodes, row sep=1.5em, column sep=2em, text height=1.5ex, text depth=0.25ex]
		{
			(U_i^{d'_i})_x		& (p^*U_i)_x		& & (p^*U_j)_x		& (p^*U_j/U_j^{d'_j})_x \\
			W_i^{d'_i}		& V_i			& & V_j			& V_j/W_j^{d'_j}, \\
		};
		\path[->, font=\scriptsize]
		(m-1-1) edge (m-1-2)
		(m-1-2) edge node[auto] {$(p^*U_\alpha)_x$} (m-1-4)
		(m-1-4) edge (m-1-5)
		(m-2-1) edge (m-2-2)
		(m-2-2) edge node[auto] {$M_\alpha$} (m-2-4)
		(m-2-4) edge (m-2-5)
		;
		\draw[-] ($(m-1-1.south) + (-.1em,0)$) -- ($(m-2-1.north) + (-.1em,.2em)$);
		\draw[-] ($(m-1-1.south) + (.1em,0)$) -- ($(m-2-1.north) + (.1em,.2em)$);
		%
		\draw[-] ($(m-1-2.south) + (-.1em,0)$) -- ($(m-2-2.north) + (-.1em,0)$);
		\draw[-] ($(m-1-2.south) + (.1em,0)$) -- ($(m-2-2.north) + (.1em,0)$);
		%
		\draw[-] ($(m-1-4.south) + (-.1em,0)$) -- ($(m-2-4.north) + (-.1em,0)$);
		\draw[-] ($(m-1-4.south) + (.1em,0)$) -- ($(m-2-4.north) + (.1em,0)$);
		%
		\draw[-] ($(m-1-5.south) + (-.1em,0)$) -- ($(m-2-5.north) + (-.1em,.2em)$);
		\draw[-] ($(m-1-5.south) + (.1em,0)$) -- ($(m-2-5.north) + (.1em,.2em)$);
	\end{tikzpicture}
\end{center}
so $(\phi_\alpha^{d'})_x$ is not the zero-map. This implies that for every point $x$ of $\Fl(U)$, there exists $\alpha \in Q_1$ such that $x$ is not contained in $Z(\phi_\alpha^{d'})$. Hence,
$
	\bigcap_{\alpha \in Q_1} Z(\phi_\alpha^{d'}) = \emptyset
$
and therefore, we obtain the following result.

\begin{prop}
	If $d' \leq d$ is $\theta$-forbidden then
	$
		\prod_{\alpha \in Q_1} \Z(\phi_\alpha^{d'}) = 0
	$
	in $A(\Fl(U))$.
\end{prop}

Let us calculate these products more explicitly. For $\alpha: i \to j$, the localized top Chern class $\Z(\phi_\alpha^{d'})$, or more precisely its push-forward to $A(\Fl(U))$, coincides with the top Chern class of the bundle
$$
	(U_i^{d'_i})^\vee \otimes p^*U_j/U_j^{d'_j}.
$$
Let $F_i^\nu := U_i^\nu/U_i^{\nu-1}$ be the successive subquotients of the universal flag $U^\bull$. With $\xi_{i,\nu} := c_1(F_i^\nu)$, we get
\begin{eqnarray*}
	c_t(U_i^{d'_i}) &=& (1 + \xi_{i,1}t)\ldots (1 + \xi_{i,d'_i}t) \\
	c_t(p^*U_j/U_j^{d'_j}) &=& (1 + \xi_{j,d'_j+1}t)\ldots (1 + \xi_{j,d_j}t)
\end{eqnarray*}
and consequently,
$$
	c_t\left((U_i^{d'_i})^\vee \otimes p^*U_j/U_j^{d'_j}\right) = \prod_{\mu = 1}^{d'_i} \prod_{\nu = d'_j+1}^{d_j} \Big( 1 + (\xi_{j,\nu} - \xi_{i,\mu})t \Big).
$$
We obtain
$$
	\Z(\phi_\alpha^{d'}) = \prod_{\mu = 1}^{d'_i} \prod_{\nu = d'_j+1}^{d_j} (\xi_{j,\nu} - \xi_{i,\mu}).
$$
The previous proposition then reads as follows:

\begin{cor}\label{ko1}
	For every $\theta$-forbidden sub-dimension vector $d' \leq d$, we have
	$$
		\prod_{\alpha: i \to j} \prod_{\mu = 1}^{d'_i} \prod_{\nu = d'_j+1}^{d_j} (\xi_{j,\nu} - \xi_{i,\mu}) = 0
	$$
	in $A(\Fl(U))$.
\end{cor}

We can view $\Fl(U)$ as an iterated formation of flag bundles. Therefore, it is easy to determine the Chow ring of $\Fl(U)$ in terms of the Chow ring of $M^\theta$ and the Chern classes $c_1(F_i^\nu)$.\\
Let $\tilde{Q}_0$ be the set of all pairs $(i,\nu)$ with $i \in Q_0$ and $1 \leq \nu \leq d_i$ and let $C$ be the polynomial ring $C := \Q[t_{i,\nu} \mid (i,\nu) \in \tilde{Q}_0]$. We define an action of the group $W := \prod_{i \in Q_0} S_{d_i}$ on $C$ by
$$
	w \cdot t_{i,\nu} = t_{i,w_i(\nu)},
$$
where $w = (w_i \mid i \in Q_0)$. Then, $A := C^W$ is generated by the algebraically independent elements $x_{i,\nu} := \sigma_\nu(t_{i,1},\ldots,t_{i,d_i})$, and $\sigma_\nu$ denotes the $\nu$-th elementary symmetric function (in the suitable number of variables).\\
Define the ring homomorphism $\Psi: C \to A(\Fl(U))$ by $\Psi(t_{i,\nu}) := \xi_{i,\nu}$. As $\sigma_\nu(\xi_{i,1},\ldots,\xi_{i,d_i}) = c_\nu(U_i)$, the map $\Psi$ restricts to $\Phi: A \to A(M^\theta)$, sending $x_{i,\nu}$ to $c_\nu(U_i)$. We get a commuting square
\begin{center}
	\begin{tikzpicture}[description/.style={fill=white,inner sep=2pt}]
		\matrix(m)[matrix of math nodes, row sep=2em, column sep=4em, text height=1.5ex, text depth=0.25ex]
		{
			C & A(\Fl(U)) \\
			A & A(M^\theta). \\
		};
		\path[->, font=\scriptsize]
		(m-1-1) edge node[auto] {$\Psi$} (m-1-2)
		(m-2-1) edge (m-1-1)
		(m-2-2) edge node[auto] {$p^*$} (m-1-2)
		(m-2-1) edge node[auto] {$\Phi$} (m-2-2);
	\end{tikzpicture}
\end{center}
Theorem \ref{groth_orig} implies at once the following fact:

\begin{thm}\label{groth}
	The homomorphism $\Psi$ induces an isomorphism $C \otimes_A A(M^\theta) \xto{}{\cong} A(\Fl(U))$ of (graded) $A(M^\theta)$-algebras.
\end{thm}

We can easily see that $C$ is a free $A$-module. For example, an $A$-basis of $C$ is given by $(t^\lambda \mid \lambda \in \Delta)$, where
$$
	t^\lambda := \prod_{i \in Q_0} \left( t_{i,1}^{\lambda_{i,1}}\ldots t_{i,d_i}^{\lambda_{i,d_i}} \right)
$$
and $\Delta$ is the set of all tuples $\lambda = (\lambda_{i,\nu} \mid (i,\nu) \in \tilde{Q}_0)$ of non-negative integers with $\lambda_{i,\nu} \leq d_i - \nu$ for all $i$ and $\nu$. This implies that $A(\Fl(U))$ is a free $A(M^\theta)$-module.\\
Let $d' \leq d$ be a $\theta$-forbidden sub-dimension vector. For an arrow $\alpha: i \to j$, let $f^{d'}_\alpha \in C$ be defined as
$$
	f^{d'}_\alpha := \prod_{\mu = 1}^{d'_i} \prod_{\nu = d'_j+1}^{d_j} (t_{j,\nu} - t_{i,\mu}).
$$

\begin{defn*}
	For a $\theta$-forbidden sub-dimension vector $d'$ of $d$, we call $f^{d'} := \prod_{\alpha \in Q_1} f^{d'}_\alpha$ the \textbf{forbidden polynomial} associated to $d'$.
\end{defn*}

Corollary \ref{ko1} shows that $\Psi(f^{d'}) = 0$. Let $\BB$ be a basis of $C$ as an $A$-module. It is of the form $\BB = ( y_\lambda \mid \lambda \in \Delta )$. There exist uniquely determined $\tau_\lambda(d',\BB) \in A$ such that
$$
	f^{d'} = \sum_{\lambda \in \Delta} \tau_\lambda(d',\BB) \cdot y_\lambda.
$$

\begin{defn*}
	The elements $\tau_\lambda(d',\BB)$ are called \textbf{tautological relations} for $d'$ with respect to $\BB$.
\end{defn*}

We will show in the next section that, together with the linear relation $l := \sum_i a_i x_{i,1}$ that we figured out earlier, the tautological relations generate the kernel of $\Phi$ and are a complete system of relations for $A(M^\theta)$.
	\section{The Chow ring of $M^{\theta}(Q,d)$ is tautologically presented}
\newcommand{\coeff}{\mathrm{coeff}}
\newcommand{\Lamdba}{\Lambda}
Let $Q$ be a connected, acyclic quiver and $d$ be a coprime dimension vector for $Q$. Let $\theta$ be a stability condition for $Q$ such that $d$ is $\theta$-coprime. Fix integers $a_i$ with $\sum_{i \in Q_0} a_i d_i = 1$. Let $M^\theta := M^{\theta}(Q,d)$ be the moduli space and let $U$ be a universal representation (as defined in section 1).\\
Let $C$, $A$ and $W$ be as in the previous section. Fix a basis $\BB = (y_\lambda \mid \lambda )$ of $C$ as an $A$-module and let $\tau_\lambda(d') := \tau_\lambda(d',\BB)$ be the tautological relations for every $\theta$-forbidden sub-dimension vector $d'$ with respect to this basis. Let $l := \sum_i a_i x_{i,1}$ be the linear relation corresponding to the integers $a_i$.

This section is devoted to the proof of the following result:

\begin{thm}\label{thm}
	The map $A \to A(M^\theta)$ sending $x_{i,\nu}$ to $c_\nu(U_i)$ yields an isomorphism $A/\aa \cong A(M^\theta)$ of graded $\Q$-algebras, where $\aa$ is the ideal generated by the linear relation $l$ and the tautological relations $\tau_\lambda(d')$ for $d'$ running through all $\theta$-forbidden sub-dimension vectors of $d$ and $\lambda \in \Delta$.
\end{thm}

\begin{rem*}
	For applications, it may be useful to compute the tautological relations $\tau(d')$ with respect to different bases for different forbidden sub-dimension vectors $d'$, i.e. we choose a basis $\BB^{d'}$ for every $d'$. A close inspection of the proof shows that the theorem remains valid under this slight generalization.
\end{rem*}

The proof proceeds in several steps. We start with some simple reductions. Remember the commuting square from above
\begin{center}
	\begin{tikzpicture}[description/.style={fill=white,inner sep=2pt}]
		\matrix(m)[matrix of math nodes, row sep=2em, column sep=4em, text height=1.5ex, text depth=0.25ex]
		{
			C & A(\Fl(U)) \\
			A & A(M^\theta). \\
		};
		\path[->, font=\scriptsize]
		(m-1-1) edge node[auto] {$\Psi$} (m-1-2)
		(m-2-1) edge (m-1-1)
		(m-2-2) edge node[auto] {$p^*$} (m-1-2)
		(m-2-1) edge node[auto] {$\Phi$} (m-2-2);
	\end{tikzpicture}
\end{center}
Evidently, $\Psi(l) = \Phi(l) = 0$. If $f^{d'} = \sum_\lambda \tau_\lambda y_\lambda$ then $0 = \Psi(f^{d'}) = \sum_\lambda \Phi(\tau_\lambda) \Psi(y_\lambda)$, and thus, Theorem \ref{groth} yields $\Phi(\tau_\lambda) = 0$. Therefore, we obtain $\Phi(\aa) = 0$, and consequently, $\Phi$ induces $\bar{\Phi}: A/\aa \to A(M^\theta)$. Theorem \ref{king2} yields that $\Phi$ is onto, so we obtain the surjectivity of $\bar{\Phi}$. Hence, it remains to prove that $\bar{\Phi}$ is injective.

Furthermore, we note that there is no loss of generality in assuming $\theta(d) = 0$. This is because neither multiplication of the stability condition with a positive integer, nor adding an integral multiple of $\dim$ changes the set of forbidden sub-dimension vectors.

The following is inspired by the proof of a result due to Ellingsrud and Str\o{}mme (cf. \cite[Theorem 4.4]{ES}). Like they do, we first prove the desired result for a torus quotient using methods of toric geometry. In our situation, this amounts to choosing the dimension vector consisting of ones only. Afterwards, we reduce the general case to a toric one. Ellingsrud and Str\o{}mme use a symmetrization map $p$ to obtain the ideal of relations. This map will also play a role in the following proof: We show that the ideal of tautological relations contains the image via $p$ of the ideal generated by the forbidden polynomials. After having done so, we proceed in almost the same way as in \cite{ES} (cf. Lemmas \ref{anti_invar}, \ref{degrees} and \ref{perfPair}).

\subsection{The toric case}

We prove that $\bar{\Phi}$ is injective when $d = \one$, the dimension vector that consists of ones entirely. A forbidden sub-dimension vector of $\one$ is of the form $\one_{I'}$, the characteristic function on a subset $I'$. Denote $\theta(I') := \theta(\one_{I'})$. Using this description of sub-dimension vectors, Theorem \ref{thm}, which we want to prove for $d = \one$, reads like this:

\begin{prop} \label{thm_toric}
	There is an isomorphism $\Q[t_i \mid i \in Q_0]/\aa \cong A(M^\theta)$ of graded $\Q$-algebras sending $t_i$ to $c_1(U_i)$, where $\aa$ is the ideal generated by functions $l = \sum_i a_i t_i$ and
	$$
		f^{I'} = \prod_{\alpha: i \to j,\ i \in I',\ j \notin I'} (t_j - t_i)
	$$
	with $I'$ running through all subsets of $Q_0$ with $\theta(I') < 0$.
\end{prop}

\begin{rem*}
	As C. Chindris has pointed out to me, Proposition \ref{thm_toric} remains valid also with integral coefficients and over (algebraically closed) fields of arbitrary characteristic.
\end{rem*}

We will prove Proposition \ref{thm_toric} by showing that $M^\theta$ is a toric variety and giving an explicit description of its toric fan. This enables us to employ a Theorem of Danilov which displays the Chow ring of a non-singular projective toric variety in terms of generators and relations.

Let $M$ be a representation of $(Q,\one)$. A sub-representation $M'$ consists of subspaces $V'_i \sub V_i \cong \kk$ such that $M_\alpha V'_i \sub V'_j$ holds for every $\alpha: i \to j$. This is equivalent to requiring $V'_j \neq 0$ for every $\alpha: i \to j$ with $M_\alpha \neq 0$ and $V'_i \neq 0$. Define the subset $I'$ (which depends on $M'$) by
$$
	I' := \{ i \in Q_0 \mid V'_i \neq 0 \}.
$$
This subset satisfies the following condition:
\begin{enumerate}[label=(\arabic{*}), ref=\arabic{*}]
	\item For every arrow $\alpha: i \to j$ with $i \in I'$ and $M_\alpha \neq 0$, it follows that $j \in I'$.\label{eqn1}
\end{enumerate}

Conversely, every subset $I' \sub Q_0$ satisfying condition (\ref{eqn1}) defines a sub-repre\-sen\-tation $M'$ of $M$. The dimension vector of this sub-representation $M'$ is $\one_{I'}$. We obtain that a representation $M$ of $(Q,\one)$ is (semi-)stable if and only if $\theta(I') > 0$ (or $\theta(I') \geq 0$) for every subset $I' \sub Q_0$ that has the property (\ref{eqn1}).\\
Let us have a look at property (\ref{eqn1}) again. It actually does not depend on $M$, but only on whether or not $M_\alpha = 0$. So, if we define
$
	\Supp(M) := \{ \alpha \in Q_1 \mid M_\alpha \neq 0 \},
$
it is clear that the (semi-)stability of $M$ only depends on the set $\Supp(M)$ of arrows.

\begin{defn*}
	A subset $J \sub Q_1$ is called $\theta$-\textbf{(semi-)stable} if there exists a $\theta$-(semi-)stable representation $M \in R$ with $J = \Supp(M)$.
\end{defn*}

We define $\SS(J)$ to be the the set of all subsets $I' \sub Q_0$ such that $j \in I'$ for every arrow $\alpha: i \to j$ with $i \in I'$ and $\alpha \in J$. We have seen:

\begin{lem} \label{lemRepSubset}
	A subset $J \sub Q_1$ is $\theta$-(semi-)stable if and only if $\theta(I') > 0$ (or $\theta(I') \geq 0$, respectively) for every $I' \in \SS(J)$.
\end{lem}

This is the simplest way to describe (semi-)stability of a representation of $(Q,\one)$. On the other hand, we can interpret $M$ as an element of the variety $R := R(Q,\one) = \bigoplus_{\alpha} R_\alpha$ with $R_\alpha = \Hom(V_i,V_j) \cong \kk$. Let us work out another characterization of (semi-)stability from this geometric point of view.

Let $T^+$ be the maximal torus of $\Gl(R)$ that corresponds to the decomposition $R = \bigoplus_\alpha R_\alpha$. Let $T := G(Q,\one) = \prod_i \G_m$. The action of $T$ on $R$ is compatible with the decomposition of $R$, thus induces a morphism $r: T \to T^+$ of tori. The kernel of this morphism $r$ is the image $\Gamma$ of the diagonal embedding $\G_m \to T$ (as we assume the quiver to be connected). This gives an embedding of $PT := T/\Gamma \to T^+$. Let $T^- := T^+/PT$. We have an exact sequence of tori
$$
	1 \to \G_m \to T \xto{}{r} T^+ \xto{}{s} T^- \to 1.
$$
This gives rise to exact sequences
$$
	0 \to \Z \to \Lambda \xto{}{r_*} \Lambda^+ \xto{}{s_*} \Lambda^- \to 0
$$
of the corresponding lattices of one-parameter subgroups and
$$
	0 \ot \Z \ot \Upsilon \xot{}{r^*} \Upsilon^+ \xot{}{s^*} \Upsilon^- \ot 0
$$
of the character lattices (cf. for example \cite{Springer:98}). Denote by $\langle \cdot,\cdot \rangle$ the pairings between the characters and the one-parameter subgroups. We have dual bases $\Upsilon = \bigoplus_i \Z \chi_i$ and $\Lambda = \bigoplus_i \Z \lambda_i$, as well as $\Upsilon^+ = \bigoplus_\alpha \Z \chi_\alpha$ and $\Lambda^+ = \bigoplus_\alpha \Z \lambda_\alpha$. Let $\delta := \sum_i \lambda_i \in \Lambda$. This is the image of $1$ under the map $\Z \to \Lambda$.

By Mumford's criterion (cf. \cite[Theorem 2.1]{GIT:94}) reformulated by King (cf. \cite[Proposition 2.6]{King:94}), we obtain the following characterization of (semi-)stability.
\begin{thm}[Mumford, King]
	For a representation $M \in R$ with support $J \sub Q_1$, the following are equivalent:
	\begin{enumerate}
		\item $M$ is $\theta$-(semi-)stable.
		\item For every $\lambda \in \Lambda_\R - \R\delta$ such that $\langle \chi_\alpha, r_*\lambda \rangle \geq 0$ for all $\alpha \in J$, we have $\langle \theta, \lambda \rangle > 0$ (or $\langle \theta, \lambda \rangle \geq 0$, respectively).
	\end{enumerate}
\end{thm}
\newcommand{\cone}{\mathrm{cone}}
The varieties $R^\theta$ and $M^\theta = R^\theta/T$ are toric: The torus $T^+$ acts on $R^\theta$ with a dense orbit isomorphic to $PT$ and therefore, $T^-$ also acts on $M^\theta$ with a dense orbit isomorphic to $T^-$. As references on toric geometry, we recommend Danilov's article \cite{Danilov} and Fulton's book \cite{Fulton:93}. An application of toric geometry to quiver moduli has been done by Hille \cite{Hille:98}. Let $\Delta^+$ be the fan of $R^\theta$ in $\Lambda^+_\R$ and $\Delta^-$ be the fan of $M^\theta$ in $\Lambda^-_\R$. We want to give an explicit description of these fans.\\
We need some auxiliary results to finally obtain this description in Proposition \ref{toric_fans}. Define
$$
	\Phi_\theta := \left\{ J \sub Q_1 \mid J^c \text{ is $\theta$-(semi-)stable } \right\},
$$
where $J^c := Q_1 - J$. Note that $\theta$-stability and $\theta$-semi-stability coincide as we have already pointed out. As every subset of $Q_1$ containing a (semi-)stable set is itself (semi-)stable, we obtain that $\Phi_\theta$ is a simplicial complex. For every $J \sub Q_1$, let $\sigma^+_J := \cone(\lambda_\alpha \mid \alpha \in J) = \{ \sum_{\alpha \in J} a_\alpha \lambda_\alpha \mid a_\alpha \geq 0 \}$ in $\Lambda^+_\R$ and let $\sigma_J^- := s_*\sigma^+_J \sub \Lambda^-_\R$.
\begin{lem} \label{simplex}
	\begin{enumerate}
		\item For every $J \in \Phi_\theta$, the cone $\sigma_J^-$ is a simplex of dimension $\sharp J$.
		\item For $J_1, J_2 \in \Phi_\theta$, we have $\sigma_{J_1}^- \cap \sigma_{J_2}^- = \sigma_{J_1 \cap J_2}^-$.
		\item If $J' \notin \Phi_\theta$ such that $\{\alpha\} \in \Phi_\theta$ holds for every $\alpha \in J'$ then $\sigma_{J'}^- \notin \{ \sigma_J^- \mid J \in \Phi_\theta \}$.
		\item If $J = \{\alpha_0\} \in \Phi_\theta$ then $s_*\lambda_{\alpha_0}$ is the minimal lattice point of $\sigma_J^-$.
	\end{enumerate}
\end{lem}
\begin{proof}
	\begin{enumerate}
		\item We have to show that the elements $s_*\lambda_\alpha$ with $\alpha \in J$ are linearly independent over the reals. Let us assume there were an element $0 \neq \lambda^+ = \sum_{\alpha \in J} b_\alpha \lambda_\alpha \in \sigma^+_J$ with $s_*\lambda^+ = 0$. Then, there exists $\lambda \in \Lambda_\R - \R\delta$ with $\lambda^+ = r_*\lambda$. For every $\alpha \notin J$, we get
		$$
			\langle \chi_\alpha, r_*\lambda \rangle = \langle \chi_\alpha, \lambda^+ \rangle = 0
		$$
		as $\lambda^+$ is supported in $J$. By stability of $J^c$, we obtain that $\langle \theta, \lambda \rangle > 0$. But, on the other hand, we also get $\langle \chi_\alpha, r_*(-\lambda) \rangle = 0$ and thus, $\langle \theta, -\lambda \rangle > 0$. A contradiction.
		\item Assume there existed $\lambda' = \sum_{\alpha \in J_1} b_\alpha' \lambda_\alpha \in \sigma^+_{J_1}$ and $\lambda'' = \sum_{\alpha \in J_2} b_\alpha'' \lambda_\alpha \in \sigma^+_{J_2}$ with $\lambda' \neq \lambda''$ and $s_*\lambda' = s_*\lambda''$. Then, there exists an element $\lambda \in \Lambda_\R - \R\delta$ such that $\lambda' - \lambda'' = r_*\lambda$. For all $\alpha \notin J_2$, we have
		$$
			\langle \chi_\alpha, r_*\lambda \rangle = b_\alpha' - b_\alpha'' = b_\alpha' \geq 0,
		$$
		and as $J_2^c$ is $\theta$-stable, we obtain that $\langle \theta, \lambda \rangle > 0$. But with the same argument, we get $\langle \chi_\alpha, r_*(-\lambda) \rangle = b_\alpha'' \geq 0$ for all $\alpha \notin J_1$ and therefore, $\langle \theta, -\lambda \rangle > 0$ by stability of $J_1^c$. Again, this is a contradiction.
		\item We assume there existed $J \in \Phi_\theta$ with $\sigma_J^- = \sigma_{J'}^-$. The set $J'$ cannot be contained in $J$, thus there exists an arrow $\alpha \in J' - J$. Consider $s_*\lambda_{\alpha} \in \sigma_{J'}^-$. As $\{\alpha\} \in \Phi_\theta$, part (i) yields $s_*\lambda_{\alpha} \neq 0$. But $s_*\lambda_{\alpha} \in \sigma^-_J \cap \sigma^-_{\{\alpha\}} = \sigma^-_\emptyset = 0$ by (ii).
		\item We will show that for every $\alpha_0 \in Q_1$, the generator $s_*\lambda_{\alpha_0}$ is either $0$ or primitive (i.e. cannot be displayed as a positive integer multiple of a lattice element, apart from itself). This proves the desired statement as $s_*\lambda_{\alpha_0} \neq 0$ if $\{\alpha_0\} \in \Phi_\theta$. Let $\alpha_0: i_0 \to j_0$ with $s_*\lambda_{\alpha_0} \neq 0$ and assume there were an integer $n > 1$ and an element $\lambda^- \in \Lambda^-$ such that $s_*\lambda_{\alpha_0} = n\lambda^-$. We find $\lambda^+ \in \Lambda^+$ with $s_*\lambda_{\alpha_0} = n\lambda^- = s_*(n\lambda^+)$. This implies that there exists $\lambda = \sum_i b_i \lambda_i \in \Lambda$ with
		$$
			\sum_{\alpha: i \to j} (b_j - b_i)\lambda_\alpha = r_*\lambda = n\lambda^+ - \lambda_{\alpha_0},
		$$
		which shows that $n$ divides $b_j - b_i$ if there exists an arrow $\alpha: i \to j$ with $\alpha \neq \alpha_0$, and also that $b_{j_0} - b_{i_0}$ is \emph{not} a multiple of $n$. Consequently, $n$ divides $b_j - b_i$ if there exists an unoriented path between $i$ and $j$ that does not involve $\alpha_0$, and $b_j - b_i$ is not a multiple of $n$ if there exists an unoriented path between $i$ and $j$ that passes through $\alpha_0$ exactly once. We distinguish two cases. If there exist vertices $i$ and $j$ (not necessarily distinct) such that two unoriented paths between $i$ and $j$ exist, one of which does not run through $\alpha_0$ and the other does exactly once, we get a contradiction. So, we are down to the case where such vertices $i$ and $j$ do not exist. In this situation, removing the arrow $\alpha_0$ splits the quiver into two disjoint sub-quivers $C^{1}$ and $C^2$ with $i_0 \in C^1$ and $j_0 \in C^2$. This means that in $Q$, there is no arrow from $C^1_0$ to $C^2_0$ apart from $\alpha_0$ and no 
arrow from $C^2_0$ to $C^1_0$. Consider the element $\lambda' := \sum_{i \in C^2_0} \lambda_i$. We obtain
		$$
			r_*\lambda' = \sum_{\alpha: i \to j} (\one_{C^2_0}(j) - \one_{C^2_0}(i)) \lambda_\alpha = \lambda_{\alpha_0}.
		$$
		In turn, this implies that $s_*\lambda_{\alpha_0} = s_*r_*\lambda' = 0$, which contradicts our assumption. \qedhere
	\end{enumerate}
\end{proof}
We need a simple algebraic result to calculate the invariant ring of the affine toric varieties $X_{\sigma^+_J}$ under the $T$-action.
\begin{lem} \label{hopf}
	Let $0 \to \Upsilon'' \xto{}{\phi} \Upsilon \xto{}{\psi} \Upsilon'$ be an exact sequence of lattices (or just abelian groups) and let $S$ be a sub-monoid of $\Upsilon$. Let $A$ be a (commutative) ring. Consider the co-action
	$$
		c: A[S] \to A[\Upsilon'] \otimes_A A[S] = A[\Upsilon' \times S],
	$$
	defined by $c(t^\chi) = t^{\psi \chi} \otimes t^\chi = t^{(\psi \chi,\chi)}$. Then, the co-invariant ring $A[S]^{A[\Upsilon']}$ equals the subring $A[\phi^{-1}S]$.
\end{lem}
\begin{proof}
	By definition, $A[S]^{A[\Upsilon']} = \{ f \in A[S] \mid c(f) = 1 \otimes f \}$. As a co-action, $c$ is compatible with multiplication in $A[S]$, thus, $A[S]^{A[\Upsilon']}$ is a subring of $A[S]$. As every $t^{\phi(\chi)}$ is $A[\Upsilon']$-co-invariant, we obtain that $A[\phi^{-1}S]$ is contained in $A[S]^{A[\Upsilon']}$. On the other hand, let $f = \sum_{\chi \in S} f_\chi t^\chi \in A[S]$ be $A[\Upsilon']$-co-invariant. That means
	$$
		\sum_{\chi \in S} f_\chi t^{(\psi \chi,\chi)} = c(f) = 1 \otimes f = \sum_{\chi \in S} f_\chi t^{(0,\chi)}.
	$$
	As the elements $t^{(\chi',\chi)}$ with $\chi' \in \Upsilon'$ and $\chi \in S$ are linearly independent, we obtain that $\psi \chi = 0$ for all $\chi \in S$ with $f_\chi \neq 0$. But then $\chi \in \im \phi$, and therefore, $\chi$ possesses an inverse image in $\phi^{-1}S$.
\end{proof}
\begin{prop} \label{toric_fans}
	We obtain $\Delta^+ = \{ \sigma^+_J \mid J \in \Phi_\theta \}$ and $\Delta^- = \{ \sigma_J^- \mid J \in \Phi_\theta \}$.
\end{prop}
\begin{proof}
	\begin{enumerate}
		\item Let $\Sigma^+ := \{ \sigma^+_J \mid J \in \Phi_\theta \}$ and let $X := X_{\Sigma^+}$ be the toric variety associated to $\Sigma^+$. Then $X$ is an open subset of $R$ which is defined as the union $X = \bigcup_{J \in \Phi_\theta} X_{\sigma^+_J}$. By definition,
		$$
			X_{\sigma^+_J} = \Spec \kk[\Upsilon^+ \cap (\sigma^+_J)^\vee] = \Big\{ \sum_{\alpha \in Q_1} M_\alpha \in R \mid M_\alpha \neq 0 \text{ for all } \alpha \notin J \Big\},
		$$
		and thus, it is clear that $R^\theta = X$.
		\item Put $\Sigma^- := \{ \sigma_J^- \mid J \in \Phi_\theta \} = s_*{\Sigma^+}$. Let $Y := X_{\Sigma^-}$ be the associated toric variety. The homomorphism $s_*: \Lambda^+ \to \Lambda^-$ yields a morphism $\eta: R^\theta = X_{\Sigma^+} \to X_{\Sigma^-} = Y$ of toric varieties. This morphism is $T^+$-equivariant via $s: T^+ \to T^-$ and therefore, $T$-invariant with respect to the induced $T$-actions via $r: T \to T^+$. We know that a geometric $T$-quotient exists. Thus, it suffices to show that $Y$ is a categorical $T$-quotient.\\
		Let $f: R^\theta \to Z$ be a $T$-invariant morphism of varieties. Let $X_J := X_{\sigma^+_J}$. This is an affine open subset of $R$. Lemma \ref{hopf} shows that
		$$
			\kk[\Upsilon^+ \cap (\sigma^+_J)^\vee]^T = \kk[\Upsilon^- \cap (s^*)^{-1}((\sigma^+_J)^\vee)] = \kk[\Upsilon^- \cap (\sigma_J^-)^\vee],
		$$
		and therefore, $\eta_J: X_J \to Y_J := X_{\sigma_J^-}$ is a \emph{universal} categorical quotient (cf. \cite[Theorem 1.1]{GIT:94}). Note that $Y_J$ is an open subset of $Y$ and $\eta_J$ coincides with the restriction of $\eta$. By the quotient property, we obtain that there exists a unique morphism $g_J: Y_J \to Z$ with $g_J \eta_J = f_J := f|X_J$. Applying Lemma \ref{simplex}(ii), we get $\sigma_{J_1}^- \cap \sigma_{J_2}^- = \sigma_{J_1 \cap J_2}^-$ for all $J_1,J_2 \in \Phi_\theta$. Let $J := J_1 \cap J_2$. For $\nu = 1,2$ we have
		$$
			g_J\eta_J = f_J = f_{J_\nu}|X_{J_\nu} = g_{J_\nu}\eta_{J_\nu}|X_J = \left(g_{J_\nu}|Y_J\right) \eta_J
		$$
		because $\eta_{J_\nu}$ is a universal categorical quotient. As $g_J$ is uniquely determined by the property $g_J \eta_J = f_J$, we obtain $g_{J_\nu}|Y_J = g_J$. This proves that the maps $g_J$ with $J \in \Phi_\theta$ glue together to a map $g: Y \to Z$ with $g\eta = f$. Conversely, every such map $g$ has to fulfill $g|Y_J = g_J$. \qedhere
	\end{enumerate}
\end{proof}

We have obtained an explicit description of the fan of $M^\theta$. This enables us to prove Proposition \ref{thm_toric} with the help of Danilov's theorem (cf. \cite[Theorem 10.8]{Danilov}).

\begin{proof}[\normalfont \textit{Proof of Proposition \ref{thm_toric}}]
	Let $\sigma^-_1,\ldots,\sigma^-_r$ be the rays of the fan $\Delta^-$. By Lemma \ref{simplex} (i), we know that these come from arrows $\alpha_1,\ldots,\alpha_r$ with $\{\alpha_i\} \in \Phi_\theta$ and part (iv) of the same lemma tells us that $s_*\lambda_{\alpha_1},\ldots,s_*\lambda_{\alpha_r}$ are their minimal lattice points. Using that $\Delta^-$ is the toric fan of $M^\theta$, Danilov's theorem implies
	$$
		A(M^\theta) \cong \Q[x_{\alpha_1}, \ldots, x_{\alpha_r}]/(\rr_1 + \rr_2),
	$$
	where $\rr_1$ is the ideal generated by all monomials $x_{\alpha_{i_1}}\ldots x_{\alpha_{i_l}}$, with $J = \{\alpha_{i_1},\ldots,\alpha_{i_l}\}$ such that $\sigma_J^- \notin \Delta^-$, and $\rr_2$ is spanned by expressions $\sum_i \langle u, s_*\lambda_{\alpha_i} \rangle x_{\alpha_i}$, where $u$ runs through $\Upsilon^-$. The above isomorphism is given by sending $x_{\alpha_i}$ to the toric Weil divisor $D_{\alpha_i}$. Consider the homomorphism $\Sym_\Q(\Upsilon^+) = \Q[\chi_\alpha \mid \alpha \in Q_1] \to \Q[x_{\alpha_1},\ldots, x_{\alpha_r}]$ mapping $\chi_{\alpha_i}$ to $x_{\alpha_i}$ and $\chi_\alpha$ to $0$ if $\{\alpha\} \notin \Phi_\theta$. We write $S(T^+)$ instead of $\Sym_\Q(\Upsilon^+)$, for brevity. The inverse image of $\rr_1$ under this map is
	$$
		\left( \chi_{\alpha_{i_1}} \ldots \chi_{\alpha_{i_l}} \mid J := \{\alpha_{i_1}, \ldots, \alpha_{i_l}\} \text{ s.t. } \sigma^-_J \notin \Delta^- \right) + \Big( \chi_\alpha \mid \{\alpha\} \notin \Phi_\theta \Big) = \left( \prod\nolimits_{\alpha \in J} \chi_\alpha \mid J \notin \Phi_\theta \right)
	$$
	by virtue of Lemma \ref{simplex} (iii) and Proposition \ref{toric_fans}. In the quotient $\Q[\chi_\alpha \mid \alpha \in Q_1]/(\chi_\alpha \mid \{\alpha\} \notin \Phi_\theta)$, the expression $\sum_i \langle u, s_*\lambda_{\alpha_i} \rangle \chi_{\alpha_i}$ can be read as
	$$
		\sum_{\alpha \in Q_1} \langle u,s_*\lambda_\alpha \rangle \chi_\alpha = \sum_{\alpha \in Q_1} \langle s^*u, \lambda_\alpha \rangle \chi_\alpha = s^*u.
	$$
	Thus, we have shown that
	\begin{eqnarray*}
		\Q[x_{\alpha_1},\ldots,x_{\alpha_r}]/(\rr_1 + \rr_2) &\cong& S(T^+)/\left( \left( \prod\nolimits_{\alpha \in J} \chi_\alpha \mid J \notin \Phi_\theta \right) + S(T^+) \cdot \Upsilon^-  \right) \\
		&\cong& S(PT) / \left( \prod\nolimits_{\alpha \in J} r^*\chi_\alpha \mid J \notin \Phi_\theta \right),
	\end{eqnarray*}
	where $S(T^+)\cdot \Upsilon^-$ denotes the ideal of $S(T^+)$ which is generated by the image $s^*(\Upsilon^-)$. The map $\Z \to \Upsilon$ sending the generator $1$ to $\sum_i a_i \chi_i$ is a right inverse of the map $\Upsilon \to \Z$ (which is induced by the diagonal embedding $\G_m \to T$) that sends every $\chi_i$ to $1$. Remember that $a_i$ are integers such that $\sum_i a_i = 1$. This yields an identification of the character lattice of $PT$ with the quotient group
	$$
		\Upsilon/\left(\Z \cdot \sum\nolimits_i a_i \chi_i \right).
	$$
	Under this identification, the natural homomorphism $S(T^+) \to S(PT)$ is the map $\Q[\chi_\alpha \mid \alpha] \to \Q[\chi_i \mid i]/(\sum_i a_i\chi_i)$ that sends $\chi_\alpha$ to the coset of $\chi_j - \chi_i$ if $\alpha: i \to j$. The above calculation shows that we have an isomorphism
	$$
		A(M^\theta) \cong \Q[\chi_i \mid i \in Q_0]/\bb,
	$$
	where $\bb$ is the ideal generated by $\sum_i a_i \chi_i$ and by the terms $\prod_{\alpha \in J,\ \alpha: i \to j} (\chi_j - \chi_i)$ for all $J \notin \Phi_\theta$. This looks already pretty similar to the ideal $\aa$ in the theorem. We show that $\aa$ and $\bb$ are in fact the same after renaming the variables $\chi_i \mapsto t_i$. Both ideals contain $l = \sum_i a_i t_i$. Let $J \sub Q_1$ with $J \notin \Phi_{\theta}$. That means $J^c$ is not $\theta$-(semi-)stable and therefore, there exists a $\theta$-forbidden $I' \sub Q_0$ with $I' \in \SS(J^c)$ by Lemma \ref{lemRepSubset}. Thus, if $\alpha: i \to j$ is an arrow with $i \in I'$ and $j \notin I'$ then $\alpha \notin J^c$. This implies
	$$
		f^{I'} = \prod_{\alpha: i \to j,\ i \in I',\ j \notin I'} (t_j - t_i) \text{ divides } \prod_{\alpha: i \to j,\ \alpha \in J} (t_j - t_i)
	$$
	and thus $\bb \sub \aa$. Conversely, consider the element $f^{I'} \in \aa$ for a $\theta$-forbidden $I' \in Q_0$. If we define $J := \{ \alpha: i \to j \mid i \in I',\ j \notin I' \}$ then $I' \in \SS(J^c)$ and $f^{I'} = \prod_{\alpha: i \to j,\ \alpha \in J} (t_j - t_i)$, so $f^{I'} \in \bb$.
\end{proof}

\subsection{The general case}

Now, proceed to the general case. Let $\tilde{Q}$ be a covering quiver of $(Q,d)$ in the sense of Reineke \cite[7.1]{Reineke:08} and Weist \cite[3.1]{Weist:09} defined by
\begin{eqnarray*}
	\tilde{Q}_0 &=& \{ (i,\nu) \mid i \in Q_0,\ 1 \leq \nu \leq d_i \} \\
	\tilde{Q}_1 &=& \{ (\alpha,\mu,\nu) \mid (\alpha: i \to j) \in Q_1,\ 1 \leq \mu \leq d_i,\ 1 \leq \nu \leq d_j \},
\end{eqnarray*}
and for every $\beta := (\alpha,\mu,\nu) \in \tilde{Q}_1$ with $\alpha: i \to j$ in $Q$, the arrow $\beta$ starts at $(i,\mu)$ and ends at $(j,\nu)$. Note that this definition of the symbol $\tilde{Q}_0$ coincides with the one we have used before. We get $R = R(Q,d) \cong R(\tilde{Q},\one)$ and $T := G(\tilde{Q},\one)$ is isomorphic to a maximal torus of $G = G(Q,d)$ in such a way that the action of $T$ on $R$ coincides with the induced $T$-action by the $G$-action on $R$.

For $\tilde{a} \in \Z^{\tilde{Q}_0}$, let $s(\tilde{a})$ be the tuple $(a_i \mid i \in Q_0)$ of integers defined by $a_i := \sum_{\nu = 1}^{d_i} \tilde{a}_{(i,\nu)}$. Define
$$
	\tilde{\theta}(\tilde{a}) := \theta(s(\tilde{a})).
$$
This is a stability condition for $\tilde{Q}$. We analyze its properties.\\
A sub-dimension vector $I'$ of $\one$ is nothing but a subset $I' \sub \tilde{Q}_0$. For a subset $I'$ of $\tilde{Q}_0$, the tuple $d' := s(I')$ consists of the integers $d'_i$ defined as the number of $\nu \in \{1,\ldots,d_i\}$ with $(i,\nu) \in I'$. Therefore, $I'$ is forbidden for $\tilde{\theta}$ if and only if the induced sub-dimension vector $d'$ of $d$ is $\theta$-forbidden. Furthermore,
$$
	f^{I'}	= \prod_{\substack{(\alpha, \mu, \nu): (i,\mu) \to (j,\nu),\\ (i,\mu) \in I',\ (j,\nu) \notin I'}} (t_{j,\nu} - t_{i,\mu}) = \prod_{\alpha: i \to j}\ \prod_{\mu:\ (i,\mu) \in I'}\ \prod_{\nu:\ (j,\nu) \notin I'} (t_{j,\nu} - t_{i,\mu})
		= f^{d'}(t_{i,w_i(\nu)} \mid i,\nu) = wf^{d'}
$$
if we define $w \in W$ such that $w_i\{1,\ldots,d'_i\} = \{ \mu \mid (i,\mu) \in I'\}$. Remember $W = \prod_i S_{d_i}$ which is, by the way, the Weyl group of $T$ in $G$. It acts on $C$ by $w\, t_{i,\nu} = t_{i,w_i(\nu)}$. Conversely, every $wf^{d'}$ is of the form $f^{I'}$, where $I' = \{ (i,w_i(\nu)) \mid i \in Q_0,\ 1 \leq \nu \leq d'_i \}$. With these identifications, the first step of the proof shows that the natural map
$$
	C/\cc \to A(\tilde{Y})
$$
is an isomorphism, where $\tilde{Y} := M^{\tilde{\theta}}(\tilde{Q},\one)$ and $\cc$ is the ideal generated by $l$ and all elements $wf^{d'}$ for $\theta$-forbidden sub-dimension vectors $d' \leq d$ and $w \in W$.

\begin{defn*}
	Let $A$ be a ring, let $M$ be a free $A$-module of finite rank and let $M'$ be a sub-$A$-module of $M$. Choose a basis $\{ m_1,\ldots,m_n\}$ of $M$, and define
	$$
		\coeff_A(M',M) := p_1(M') + \ldots + p_n(M'),
	$$
	where $p_i: M \to A$ is the $A$-linear map defined by $p_i(\sum_j a_j m_j) = a_i$. We can easily see that $\coeff_A(M',M)$ is an ideal of $A$ that does not depend on the choice of a basis of $M$.
\end{defn*}

If $M' = A \cdot m$ with $m = \sum_i a_im_i$ then $\coeff_A(M',M)$ is the ideal of $A$ generated by $a_1,\ldots,a_n$. This explains the notation.

We get $\aa = \coeff_A(\cc,C)$ with the following argument: By definition, $\aa = Al + \sum_{d'} \coeff_A(f^{d'},C)$, where the sum runs over all $\theta$-forbidden sub-dimension vectors $d' \leq d$. Let $\BB = (y_\lambda \mid \lambda \in \Delta)$ be a basis of $C$ as an $A$-module and $f^{d'} = \sum_\lambda \tau_\lambda y_\lambda$. For every $w \in W$, we get
$$
	wf^{d'} = \sum_\lambda \tau_\lambda wy_\lambda
$$
and $(wy_\lambda \mid \lambda)$ is also a basis of $C$ over $A$. We obtain that $\coeff_A(wf^{d'},C) = \coeff_A(f^{d'},C)$ and therefore,
$$
	\coeff_A(\cc,C) = Al + \sum_{d'} \sum_{w \in W} \coeff_A(wf^{d'},C) = Al + \sum_{d'} \coeff_A(f^{d'},C) = \aa.
$$
Let us summarize the situation in a picture. With $Y := M^\theta(Q,d)$, we have a commutative diagram
\begin{center}
	\begin{tikzpicture}[description/.style={fill=white,inner sep=2pt}]
		\matrix(m)[matrix of math nodes, row sep=2em, column sep=2em, text height=1.5ex, text depth=0.25ex]
		{
			0 & \cc & C &[2em] \CC := A(\tilde{Y}) & 0 \\
			0 & \aa = \coeff_A(\cc,C) & A &[2em] \AA := A(Y) & 0 \\
		};
		\path[->, font=\scriptsize]
		(m-1-1) edge (m-1-2)
		(m-1-2) edge (m-1-3)
		(m-1-3) edge (m-1-4)
		(m-1-4) edge (m-1-5)
		(m-2-1) edge (m-2-2)
		(m-2-2) edge (m-2-3)
		(m-2-3) edge (m-2-4)
		(m-2-4) edge (m-2-5)
		(m-2-3) edge (m-1-3)
		(m-2-2) edge (m-1-2);
	\end{tikzpicture}
\end{center}
and we know that the upper row is exact. In order to prove the exactness of the lower row, we have remarked earlier that it suffices to check that $A/\aa \to \AA$ is injective.

Define the discriminant of $C$ by
$$
	\Delta = \prod_{i \in Q_0}\ \prod_{1 \leq \mu < \nu \leq d_i} (t_{i,\nu} - t_{i,\mu}).
$$
This is an anti-invariant, i.e. $W$ acts on $\Delta$ by $w\Delta = \left( \prod_i \sign(w_i)\right) \Delta$. It is a basic fact that every anti-invariant $y$ of $C$ is of the form $y = a\Delta$ for some $a \in A$ (see for example \cite[Chapter V, \S 4, Proposition 5]{Bourbaki:Lie_4-6}). Furthermore, define an $A$-linear map $p: C \to A$, called the symmetrization map, by
$$
	p(f) = \Delta^{-1}\sum_{w \in W} \sign(w) wf.
$$
Note that $p(\Delta) = \sharp W$ regarded as an element of $A$. We show that $\aa$ contains the image $p(\cc)$. Let $f \in \cc$. Display $f$ as a linear combination $f = \sum_i a_iy_i$ with respect to a basis $y_1,\ldots,y_n$ of $C$ as an $A$-module. As $p$ is $A$-linear, we get $p(f) = \sum_i a_ip(y_i)$ which lies in $\coeff_A(\cc,C) = \aa$ as every $a_i$ is an element of $\coeff_A(\cc,C)$.

To prove Theorem \ref{thm}, it then suffices to show that the induced map $A/p(\cc) \to \AA$ is injective. A couple of lemmas will do the trick.

\begin{lem} \label{anti_invar}
	There is a natural isomorphism $\CC^\mathrm{a} \xto{}{\cong} A/p(\cc)$ of $\Q W$-modules decreasing degrees by $\delta := \deg \Delta$.
\end{lem}

In the lemma, $\CC^\mathrm{a}$ denotes the anti-invariant part of $\CC$ as a $\Q W$-module. As the action of $W$ is compatible with the grading of $\CC = \bigoplus_{i \geq 0} \CC_i$, the anti-invariant part is a graded subspace of $\CC$, i.e. $\CC^\mathrm{a} = \bigoplus_{i \geq 0} \CC_i^\mathrm{a}$.

\begin{proof}
	The composition $C \xto{}{p} A \to A/p(\cc)$ induces an onto map $\bar{p}: \CC \cong C/\cc \to A/p(\cc)$. As taking anti-invariants is an exact functor of $\Q W$-modules, we obtain
	$$
		\CC^\mathrm{a} \cong (C/\cc)^\mathrm{a} = C^\mathrm{a}/\cc^{\mathrm{a}} = A\cdot\Delta/(A\cdot\Delta \cap \cc).
	$$
	But $A\cdot\Delta \cap \cc = p(\cc)\cdot\Delta$ because of the following: On the one hand, if $f \in \cc$ is of the form $f = a\Delta$ for some $a \in A$ then $a = p(\frac{1}{\sharp W}a\Delta) \in p(\cc)$. On the other hand, let $f = p(y)\Delta$ for a $y \in \cc$. As the ideal $\cc$ is $W$-invariant, we get $wy \in \cc$ for every $w$. This implies $p(y) \in p(\cc)$ and thus, $f \in \cc$. We deduce
	$$
		\CC^\mathrm{a} \cong A\cdot\Delta/(p(\cc)\cdot\Delta) \cong A/p(\cc)
	$$
	as $\Q W$-modules. It is clear that this isomorphism decreases degrees by $\delta$.
\end{proof}

\begin{lem} \label{degrees}
	For $i < \delta$ or $i > \dim \tilde{Y} - \delta$, we have $\CC_i^\mathrm{a} = 0$.
\end{lem}

\begin{proof}
	As every anti-invariant of $\CC$ is divisible by $\Delta$, there can be no anti-invariant of degree less than $\delta = \deg \Delta$.\\
	By Poincar\'{e} duality, the multiplication on $\CC$ induces a perfect pairing $\CC_i \otimes_\Q \CC_{l - i} \to \CC_l = \Q$ for every $i$, where $l := \dim \tilde{Y}$. As the $W$-action on the Chow ring comes from the $W$-action on the moduli space, this pairing is also $W$-equivariant. Therefore, we obtain
	$$
		\CC_i^\mathrm{a} \cong (\CC_{l-i}^\vee)^\mathrm{a} \cong (\CC_{l-i}^\mathrm{a})^\vee
	$$
	as taking anti-invariants commutes with taking duals. If $i > l - \delta$ then $l - i < \delta$, so the assertion is proved.
\end{proof}

Combining the preceding two lemmas, we obtain that $(A/p(\cc))_i = 0$ if $i < 0$ or $i > r := l -2\delta$. Note also that $r$ is precisely the dimension of $Y$, because $2\delta + \dim T = \dim G$.

\begin{lem} \label{perfPair}
	We have $(A/p(\cc))_r \cong \Q$ and the ring multiplication induces a perfect pairing $$(A/p(\cc))_i \otimes_\Q (A/p(\cc))_{r-i} \to (A/p(\cc))_r \cong \Q.$$
\end{lem}

\begin{proof}
	We know that
	$$
		(A/p(\cc))_r = (A/p(\cc))_{l - 2\delta} \cong \CC_{l-\delta}^\mathrm{a} \cong (\CC_{\delta}^\mathrm{a})^\vee \cong (A/p(\cc))_0^\vee \cong \Q
	$$
	because $\Delta \notin \cc$. We show that the pairing $(A/p(\cc))_i \otimes_\Q (A/p(\cc))_{r-i} \to \Q$ is perfect. Let $x \in (A/p(\cc))_i$ with $x \neq 0$. There is a unique alternating $y \in \CC_{i+\delta}$ with $\bar{p}(y) = x$. As the pairing $\CC_{i+\delta} \otimes_\Q \CC_{l-i-\delta} \to \Q$ is perfect and $W$-equivariant, we obtain that there exists an \emph{alternating} $y' \in \CC_{l-i-\delta}$ with $yy' \neq 0$. Choose representatives $\dot{y} \in C_{i+\delta}^\mathrm{a}$ and $\dot{y}' \in C_{l-i-\delta}^\mathrm{a}$ of $y$ and $y'$ and define $\dot{x} := p(\dot{y})$ and $\dot{x}' := p(\dot{y}')$. There are unique $a,a' \in A$ such that $a\Delta = \dot{y}$ and $a'\Delta = \dot{y}'$. We get
	$$
		\dot{x}\dot{x}' = p(\dot{y})p(\dot{y}') = (\sharp W)^2\cdot aa'  = \sharp W \cdot p(aa'\Delta).
	$$
	As $yy'$ is non-zero, the anti-invariant $aa'\Delta$ is not contained in $\cc^\mathrm{a}$. Hence, $p(aa'\Delta) \notin p(\cc)$ because the isomorphism $p: C^\mathrm{a} \to A$ maps $\cc^\mathrm{a}$ onto $p(\cc)$. Therefore, $xx' \neq 0$.
\end{proof}

Finally, we are able to show that the map $A/p(\cc) \to \AA$ is injective.

\begin{proof}[\normalfont \textit{Proof of Theorem \ref{thm}.}]
	Call this map $\beta$. We have seen in the above lemmas that $(A/p(\cc))_i$ vanishes for $i > r = \dim Y$. In addition, $\beta$ maps the degree $r$ part of $A/p(\cc)$ isomorphically onto $\AA_r \cong \Q$. Now let $x \in A/p(\cc)$ with $x \neq 0$. Without loss of generality, we assume $x$ is homogeneous of degree $i$. By Lemma \ref{perfPair}, there exists $x' \in (A/p(\cc))_{r-i}$ with $xx' \neq 0$. Then $0 \neq \beta(xx') = \beta(x)\beta(x')$, which implies $\beta(x) \neq 0$.
\end{proof}
	\section{Examples}
We now turn to some classes of quiver settings where we can compute the Chow ring more explicitly.
\subsection{The canonical stability condition and the bipartite case}

Let $\theta$ be the stability condition of $Q$ which comes from the character $G \xto{}{\mathrm{act}} \Gl(R) \xto{}{\det} \G_m$. Note that $\theta$ depends on $d$. It is of the form
$$
	\theta(a) = \sum_{i \in Q_0} a_i \cdot \left( \sum_{\alpha: j \to i} d_j - \sum_{\alpha:i \to k} d_k \right) = \langle a,d \rangle - \langle d,a \rangle,
$$
where $\langle \cdot,\cdot \rangle$ denotes the Euler form of the quiver $Q$. This is a reasonable stability condition in many examples. In fact, the examples in the following section will all use this stability condition. It is called the \textbf{canonical stability condition} of $(Q,d)$. It is evident that $\theta(d) = 0$. Anyway, $d$ is not necessarily $\theta$-coprime for this choice of $\theta$. Let us assume it is.

We define a partial ordering on the set of dimension vectors of $Q$. We write $d' \leqslant d''$ if $d'_i \leq d''_i$ for every source $i$ of $Q$, $d'_j \geq d''_j$ if $j$ is a sink of $Q$, and $d'_k = d''_k$ for every vertex $k$ of $Q$ which is neither a source nor a sink.\\
For two sub-dimension vectors $d'$ and $d''$ with $d' \leqslant d''$, the polynomial $f^{d'}$ divides $f^{d''}$. A simple calculation shows that in this case $\coeff_A(f^{d'},C)$ contains $\coeff_A(f^{d''},C)$.\\
Let $d'$ and $d''$ be two sub-dimension vectors of $d$ with $d' \leqslant d''$. Then,
\begin{eqnarray*}
	\theta(d'') - \theta(d') &=& \sum_i (d''_i - d'_i) \left( \sum_{\alpha: j \to i} d_j - \sum_{\alpha:i \to k} d_k  \right) \\
	&=& \sum_{i \text{ source}} (d'_i - d''_i) \sum_{\alpha:i \to k} d_k + \sum_{i \text{ sink}} (d''_i - d'_i) \sum_{\alpha: j \to i} d_j\quad \leq\quad 0,
\end{eqnarray*}
and thus $\theta(d') \geq \theta(d'')$. So, if $d'$ is $\theta$-forbidden, $d''$ will be ``even more forbidden''. This shows that we can restrict ourselves to the \emph{minimal} $\theta$-forbidden sub-dimension vectors $d'$ of $d$. This substantially reduces the computational effort, in particular if $Q$ is bipartite.

Let $Q$ be a bipartite quiver, let $d$ be a coprime dimension vector and let $\theta$ be the canonical stability condition. Assume that $d$ is $\theta$-coprime. Let $A$ be as usual. Under these circumstances, Theorem \ref{thm} reads like this:

\begin{cor} \label{remMinimal}
	The Chow ring $A(M^\theta)$ is isomorphic to $A/\aa$, where $\aa$ is generated by the linear relation $l$ and the tautological relations $\tau_\lambda(d')$. Here, $\lambda \in \Delta$ and $d'$ runs through all \emph{minimal} $\theta$-forbidden sub-dimension vectors of $d$.
\end{cor}

\subsection{Subspace quivers}
Let $(Q,d)$ be the following quiver setting
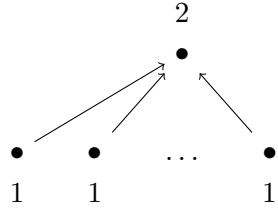
\begin{center}
	\begin{tikzpicture}[description/.style={fill=white,inner sep=2pt}]
	\matrix(m)[matrix of math nodes, row sep=2em, column sep=1.5em, text height=1.5ex, text depth=0.25ex]
	{
			&		& \bullet	& \\
		\bullet	& \bullet	& \ldots	& \bullet \\
	};
	\node[above] at (m-1-3.north) {$2$};
	\node[below] at (m-2-1.south) {$1$};
	\node[below] at (m-2-2.south) {$1$};
	\node[below] at (m-2-4.south) {$1$};
	\path[->, font=\scriptsize]
	(m-2-1) edge (m-1-3)
	(m-2-2) edge (m-1-3)
	(m-2-4) edge (m-1-3);
\end{tikzpicture}
\end{center}
with sources $q_1,\ldots,q_m$, where $m = 2r+1$ is an odd number, and sink $s$. Let the stability condition $\theta$ be defined by
$$
	\theta(a) = ma_s - 2(a_{q_1} + \ldots + a_{q_m})
$$
for $a \in \Z^{Q_0}$. This is the canonical stability condition for $(Q,d)$ described above. As $m$ is an odd number, it is immediate that $d$ is $\theta$-coprime.

The moduli space $M^\theta := M^\theta(Q,d)$ can easily be identified with the space of $m$ ordered points on the projective line, of which no more than $r$ coincide, up to the natural $\PGl_2$-action.

We want to calculate the ring $A(M^\theta) = A/\aa$ from Theorem \ref{thm}. We have $C = \Q[x_1,\ldots,x_m,y_1,y_2]$ if we rename $t_{q_i,1} =: x_i$ and $t_{s,j} =: y_j$, for convenience. Then, $A$ is the subring $\Q[x_1,\ldots,x_m,y,z]$, where $y = y_1 + y_2$ and $z = y_1y_2$. The group $W$ is just $S_2$ that acts on $C$ by swapping $y_1$ and $y_2$. As in the general setting, we fix some integers $a_i$ and $b$ such that $\sum_{i=1}^m a_i + 2b = 1$, thus
$$
	l = \sum_{i=1}^m a_ix_i + by \in A.
$$

By Corollary \ref{remMinimal}, we have to find the minimal forbidden sub-dimension vectors of $d'$. We distinguish three cases (one of which does not appear at all).
\begin{itemize}
	\item[(a)] $d'_s = 0$. Then, $d'$ is minimal forbidden if and only if there exists exactly one index $1 \leq i \leq m$ such that $d'_{q_i} = 1$.
	\item[(b)] $d'_s = 1$. We obtain $\theta(d') = m -  2 \sharp \{ i \mid d'_{q_i} = 1 \}$, which is negative if and only if $\sharp \{ i \mid d'_{q_i} = 1 \} > r$. Moreover, $d'$ is minimal forbidden if and only if the set $\{ i \mid d'_{q_i} = 1\}$ has precisely $r+1$ elements.
	\item[(c)] $d'_s = 2$. In this case, $\theta(d') \geq 0$.
\end{itemize}

Now we turn to the elements $f^{d'}$ for some fixed minimal forbidden sub-dimension vector $d' \leq d$. Distinguish again between the cases (a) and (b).
\begin{itemize}
	\item[(a)] If $d'_s = 0$ then there exists a unique $i$ with $d'_{q_i} = 1$. Then, $f^{d'} = (y_1 - x_i)(y_2 - x_i) = z - yx_i + x_i^2$. In particular, $f^{d'} \in A$.
	\item[(b)] In case $d'_s = 1$, the set $I$ of all $i$ with $d'_{q_i} = 1$ has exactly $r+1$ elements. Let $I = \{i_1,\ldots,i_{r+1}\}$. We obtain
	$$
		f^{d'} = \prod_{\nu = 1}^{r+1} (y_2 - x_{i_\nu}) = \sum_{j=0}^{r+1} (-1)^j \sigma_j^{(I)} y_2^{{r+1}-j}
	$$
	where $\sigma_j^{(I)}$ denotes the $j$-th elementary symmetric function on the variables $x_{i_1},\ldots,x_{i_{r+1}}$.
\end{itemize}
We know that $C$ decomposes as $C = A\cdot y_2 \oplus A \cdot 1$. By definition, $y_1$ and $y_2$ are roots of the polynomial $p(t) = t^2 - yt + z \in A[t]$. We want to give a presentation of $y_2^k$ as a linear combination of $y_2$ and $1$ with coefficients in $A$. Let
$$
	y_2^k = \alpha_k y_2 + \beta_k
$$
for all $k \geq 0$. Then, we obtain that $(\alpha_k)$ and $(\beta_k)$ are sequences $(\gamma_k)$ of homogeneous elements in $A$ that fulfill the recursion
$$
	\gamma_k = y\gamma_{k-1} - z\gamma_{k-2}
$$
for $k \geq 2$, but with different initial values $(\alpha_0, \alpha_1) = (0,1)$ and $(\beta_0,\beta_1) = (1,0)$. The sequences $(\gamma_k)$ satisfying the two-term recursion above can be described with linear algebra methods: Consider the matrix
$$
	M = \begin{pmatrix} 0 & 1 \\ -z & y  \end{pmatrix}.
$$
Let $v^{(k)} := (\gamma_k,\gamma_{k+1})^T$ for all $k \geq 0$. Then obviously $v^{(k)} = M^k \cdot v^{(0)}$.

Back to our relations. We get for $d'$ and $I$ as in case (b) above
$$
	f^{d'} = \left(\sum_{j = 0}^{r+1} (-1)^j \sigma_j^{(I)} \alpha_{{r+1}-j} \right)y_2 + \left(\sum_{j=0}^{r+1} (-1)^j \sigma_j^{(I)} \beta_{{r+1}-j} \right).
$$
As $d'$ is forbidden, we obtain that $f^{d'} = 0$ and this implies
$$
	0 = \sum_{j = 0}^{r+1} (-1)^j \sigma_j^{(I)} \alpha_{{r+1}-j} \quad \text{and} \quad
	0 = \sum_{j=0}^{r+1} (-1)^j \sigma_j^{(I)} \beta_{{r+1}-j}
$$
in $A/\aa$. Thus, the ideal $\aa$ from Theorem \ref{thm} is generated by the expressions
\begin{itemize}
	\item[(l)] $\sum_{i=1}^m a_ix_i + by$,
	\item[(a)] $z - yx_i + x_i^2$ for all $1 \leq i \leq m$,
	\item[(b1)] $\sum_{j = 0}^{r+1} (-1)^j \sigma_j^{(I)} \alpha_{{r+1}-j}$, and
	\item[(b2)] $\sum_{j = 0}^{r+1} (-1)^j \sigma_j^{(I)} \beta_{{r+1}-j}$ for all $I \sub \{1,\ldots,m\}$ with ${r+1} := \sharp I > r$.
\end{itemize}
Here, $a_i$ and $b$ are integers such that $a_1 + \ldots a_m + 2b = 1$. Theorem \ref{thm} states that there exists an isomorphism $\Q[x_1,\ldots,x_m, y,z]/\aa \cong A(M^\theta)$ sending $x_i \mapsto c_1(U_{q_i})$, $y \mapsto c_1(U_s)$, and $z \mapsto c_2(U_s)$, where $U$ is the universal representation corresponding to the character of weight one according to the integers $a_i$ and $b$.

We see that the presentation of the Chow ring depends essentially on the choice of a character of weight one. In this case, there are two ``reasonable choices''.

Let $a_i = \delta_{i,m}$ the Kronecker symbol and let $b = 0$. Then, the generator $l = \sum_i a_i x_i + by$ is just $l = x_m$. We consider the image of $z - yx_i + x_i^2$ in the quotient $A/A\cdot l$. For $i = m$, we obtain that $z - yx_i + x_i^2$ maps to $z$. Therefore, we have an isomorphism of $A(M^\theta)$ to $\Q[x_1,\ldots,x_{m-1},y]/\aa$, where $\aa$ is generated by expressions
\begin{itemize}
	\item[(a)] $yx_i - x_i^2$ for all $1 \leq i \leq m-1$,
	\item[(b1)] $\sum_{j = 0}^{r+1} (-1)^j \sigma_j^{(I)} \alpha_{{r+1}-j}$, and
	\item[(b2)] $\sum_{j = 0}^{r+1} (-1)^j \sigma_j^{(I)} \beta_{{r+1}-j}$ for all $I \sub \{1,\ldots,m\}$ with $r+1$ elements.
\end{itemize}
Here, we have (formally) $x_m := 0$ and $(\alpha_k)$, $(\beta_k)$ fulfill the recursions $\alpha_k = y\alpha_{k-1}$ and $\beta_k = y\beta_{k-1}$ for all $k \geq 2$. As $\beta_1 = 0$, we obtain that all $\beta_k$ vanish if $k \geq 1$, so
$
	\beta_k = \delta_{k,0}.
$
As $\alpha_1 = 1$, we get $\alpha_k = y^{k-1}$ for $k \geq 1$. By definition, we have $\alpha_0 = 0$ and thus,
$
	\alpha_k = (1 - \delta_{k,0})y^{k-1}.
$
Therefore, the expressions (b1) and (b2) read as follows: For all subsets $I \sub \{1,\ldots,m\}$ with ${r+1}$ elements, we have
\begin{itemize}
	\item[(b1)] $\sum_{j = 0}^{r} (-1)^j \sigma_j^{(I)} y^{r-j}$ and
	\item[(b2)] $\prod_{i \in I} x_i$
\end{itemize}
lying in $\aa$. Note that in (b2), we have omitted a sign. If $m \in I$, we can reformulate (b1). Let $I' := I - \{m\}$. We obtain $\sigma_j^{(I)} = \sigma_j^{(I')}$ for all $j \leq r$ and thus,
$$
	\sum_{j = 0}^{r} (-1)^j \sigma_j^{(I)} y^{r-j} = \sum_{j = 0}^{r} (-1)^j \sigma_j^{(I')} y^{r-j} = \prod_{i \in I'} (y-x_i).
$$
This implies that for $m \notin I$, the product $\prod_{i \in I} (y - x_i)$ is also contained in the ideal $\aa$. But we can rewrite this as
$$
	\prod_{i \in I} (y - x_i) = \sum_{j=0}^{r+1} (-1)^j \sigma_j^{(I)} y^{{r+1}-j} = (-1)^{r+1} \prod_{i \in I} x_i + y \cdot \left( \sum_{j=0}^{r} (-1)^j \sigma_j^{(I)} y^{r-j} \right)
$$
and thereby, the expressions of type (b2) are already contained in $\aa$ if all of type (b1) are. In summary, we get the following presentation of $A(M^\theta)$:

\begin{cor} \label{firstPres_subspace}
	The ring $A(M^\theta)$ is isomorphic to $\Q[x_1,\ldots,x_{m-1},y]/\aa$, where $\aa$ is the ideal generated by the expressions
	\begin{itemize}
		\item[(a)] $x_i(y-x_i)$ for all $1 \leq i \leq m-1$,
		\item[(b1')] $\prod_{i \in I'} (y-x_i)$, and
		\item[(b1'')] $\sum_{j = 0}^{r} (-1)^j \sigma_j^{(I)} y^{r-j}$
	\end{itemize}
	for all $I', I \sub \{1,\ldots,m-1\}$ with $\sharp I' = r$ and $\sharp I = r+1$.
\end{cor}

This presentation is similar to the one in the paper of Hausmann and Knutson (cf. \cite{HK:98}). We describe it briefly: Let $\mathbb{S}$  be the direct product of $m$ copies of the $2$-sphere. Consider the diagonal action of $\mathrm{SO}_3$ on $\mathbb{S}$. Identifying $\mathfrak{so}_3^\vee$ with $\R^3$, this action has the moment map $\mu(p) = \sum_{i = 1}^m p_i$. The symplectic reduction $\mathrm{Pol} := \mu^{-1}(0)/\mathrm{SO}_3$ is called the polygon space. This space can be identified with $M^\theta(\C)$ regarded as a symplectic manifold.

\begin{thm}[Hausmann-Knutson] \label{hk}
	The cohomology ring $H^*(\mathrm{Pol})$ is isomorphic to the quotient $\Z[V_1,\ldots,V_{m-1},R]/\II$ of a polynomial ring with generators in degree $2$, and the ideal $\II$ is generated by the expressions
	\begin{itemize}
		\item[(R1)] $V_i^2 + RV_i$ for $i = 1,\ldots,m-1$,
		\item[(R2)] $\prod\limits_{i \in L} V_i$ for all subsets $L \sub \{1,\ldots,m-1\}$ with $\sharp L \geq r$ and
		\item[(R3)] $\sum\limits_{S \sub L,\ \sharp S \leq r-1} R^{\sharp(L-S)-1}\prod\limits_{i \in S} V_i$ for all subsets $L \sub \{1,\ldots,m-1\}$ with $\sharp L \geq r+1$.
	\end{itemize}
\end{thm}

We should remark that this is the equal weight version of the main result \cite[Theorem 5.5]{HK:98}. In the actual theorem, a description of the cohomology ring of the polygon space associated to an $m$-tuple of positive weights is given. This result is proved by embedding the polygon space into a transverse self-intersection of a toric subvariety in some toric variety.\\
Although the presentation in Theorem \ref{hk} strongly resembles the presentation from Corollary \ref{firstPres_subspace}, it is hard to show that these two rings are isomorphic (over the rationals with a doubling of the grades).

The other ``reasonable choice'' is the following: We consider the numbers $a_i = 1$ for all $i$ and $b = -r$. This choice comes from viewing this subspace quiver as some covering quiver of a generalized Kronecker quiver with dimension vector $(2,2r+1)$.\\
Let $\sigma_j = \sigma_j(x_1,\ldots,x_m)$ for $j = 1,2$. The expression (l) reads $\sigma_1 - ry = 0$. Replacing $y$ with $1/r \cdot \sigma_1$, we get for (a)
$$
	z - \frac{1}{r}\sigma_1x_i + x_i^2 = 0,
$$
and summing over all $i$ yields
$$
	0 = mz - \frac{1}{r}\sigma_1^2 + \sigma_1(x_1^2,\ldots,x_m^2) = mz - \frac{1}{r}\sigma_1^2 + (\sigma_1^2 - 2\sigma_2),
$$
or equivalently
$$
	z = \frac{1-r}{mr}\sigma_1^2 + \frac{2}{m}\sigma_2.
$$
We get:

\begin{cor}
	The ring $A(M^\theta)$ is isomorphic to $\Q[x_1,\ldots,x_m]/\aa$, where $\aa$ is the ideal generated by the expressions
	\begin{itemize}
		\item[(a)] $(1-r)\sigma_1^2 + 2r\sigma_2 - m\sigma_1x_i + mrx_i^2$ for all $1 \leq i \leq m$,
		\item[(b1)] $\sum_{j = 0}^{r+1} (-1)^j \sigma_j^{(I)} \alpha_{{r+1}-j}$, and
		\item[(b2)] $\sum_{j = 0}^{r+1} (-1)^j \sigma_j^{(I)} \beta_{{r+1}-j}$
	\end{itemize}
	for all $I \sub \{1,\ldots,m\}$ with ${r+1}$ elements and $(\alpha_k)$ and $(\beta_k)$ are sequences $(\gamma_k)$ that satisfy
	$$
		\gamma_k = \frac{1}{r}\sigma_1\gamma_{k-1} - \left( \frac{1-r}{mr}\sigma_1^2 + \frac{2}{m}\sigma_2 \right)\gamma_{k-2},
	$$
	and $(\alpha_0,\alpha_1) = (0,1)$, $(\beta_0,\beta_1) = (1,0)$.
\end{cor}

This might look more frightening than the presentation calculated before, but the latter has the advantage of preserving the natural $S_m$-action on the moduli space.

The first presentation can be used to read off the Poincar\'{e} polynomial of $\AA := A(M^\theta)$. We observe that these defining relations are homogeneous of degree 2 for type (a) and of degree $\geq r$ for the other types (b1') and (b1''). We know that
$$
	\dim M^\theta = \dim R - \dim G + \dim \Gamma = 2m - (m+4) +1 = m-3 = 2(r-1).
$$
By Poincar\'{e} duality, $\AA^i \cong \AA^{2(r-1)-i}$ for every $0 \leq i \leq r-1$, so we just have to determine the dimensions of $\AA^0,\ldots,\AA^{r-1}$. But in these degrees, $\AA$ coincides with the quotient of the polynomial ring
$$
	\Q[x_1,\ldots,x_{m-1},y]/(yx_i - x_i^2 \mid 1 \leq i \leq m).
$$
Now, we verify almost at once that $yx_1 - x_1^2,\ldots,yx_{m-1} - x_{m-1}^2$ is a regular sequence, and thus, the Poincar\'{e} series of the latter ring computes as
$$
	\frac{(1-t^2)^{m-1}}{(1-t)^m} = \frac{(1+t)^{m-1}}{1-t} = \sum_{j=0}^{\infty} \left( \sum_{\nu = 0}^{\min\{j,m-1\}} \binom{m-1}{\nu} \right) t^j = \sum_{j=0}^{\infty} \left( \sum_{\nu = 0}^{\min\{j,2r\}} \binom{2r}{\nu} \right) t^j.
$$
This proves a result of Kirwan (cf. \cite[16.1]{Kirwan:84}):

\begin{cor}
	The Poincar\'{e} polynomial of $A(M^\theta)$ is
	$$
		\sum_{j=0}^{2(r-1)} \left( \sum_{\nu = 0}^{\min\{j,2(r-1)-j\}} \binom{2r}{\nu} \right) t^j.
	$$
\end{cor}
\subsection{Generalized Kronecker quivers}
Let $Q$ be the generalized Kronecker quiver with $r$ arrows $\alpha_1,\ldots,\alpha_r$ pointing from the source $q$ to the sink $s$. Let $d = (m,n)$ with coprime positive integers $n$ and $m$. In a picture:
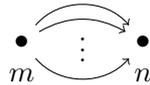
\begin{center}
	\begin{tikzpicture}[description/.style={fill=white,inner sep=2pt}]
	\matrix[column sep={3em}]
	{
		\node(a) {$\bullet$}; & \node(b) {$\bullet$}; \\
	};
	\path[->] (a) edge[bend left=50] (b)
		edge[bend left=30] (b)
		edge[bend right=50] (b);
	\node at ($(a) + (2.1em,0em)$) {$\vdots$};
	\node[below] at (a.south) {$m$};
	\node[below] at (b.south) {$n$};
\end{tikzpicture}
\end{center}
Without loss of generality, we may assume that $m \leq n$ as the moduli space of the quiver with reverted arrows is isomorphic to the original moduli space (with an obvious modification of the stability condition). The canonical stability condition $\theta$ is given by
$$
	\theta(m',n') = r\cdot(mn' - nm').
$$
As $m$ and $n$ are coprime, $d$ is $\theta$-coprime. The moduli space $M^\theta$ is often denoted $N(r; m,n)$ and called a Kronecker module. Their cohomology has already been studied by Drezet \cite[III-VI]{Drezet:88} and also by Ellingsrud and Str\o{}mme \cite[Section 6]{ES}.

The ring $C$ is, in this case, given by $C = \Q[t_1,\ldots,t_m,s_1,\ldots,s_n]$, where $t_i := t_{q,i}$ and $s_j := t_{s,j}$. Furthermore, $A = \Q[x_1,\ldots,x_m,y_1,\ldots,y_n]$, where $x_i$ is the $i$-th elementary symmetric polynomial in the variables $t_1,\ldots,t_m$ and $y_j$ is the $j$-th in $s_1,\ldots,s_n$. The Weyl group $W$ is the product $W = S_m \times S_n$ acting separately on the $t_i$ and the $s_j$ in the usual manner. Choose integers $a$ and $b$ with $am + bn = 1$. Then
$$
	l = ax_1 + by_1.
$$

Let us work out the forbidden sub-dimension vectors. We see that a sub-dimension vector $d' = (m',n')$ of $d$ is forbidden if and only if
$$
	n' < \left\lceil \frac{nm'}{m} \right\rceil.
$$
The corresponding $f^{d'} \in C$ computes as
$$
	f^{d'} = \prod_{i = 1}^{m'} \prod_{j=n'+1}^n (s_j - t_i)^r.
$$
Moreover, the forbidden sub-dimension vector $d'$ is minimal if and only if $n' = \left\lceil {nm'}/{m} \right\rceil - 1$.
Fix a minimal forbidden sub-dimension vector $d' = (m',n')$. Note that different values for $m'$ yield different values of the expression $\left\lceil {nm'}/{m} \right\rceil$, as $m \leq n$ by assumption. There are $m$ different polynomials
$$
	f^{(m')} := f^{d'} = \prod_{i = 1}^{m'} \prod_{j=\left\lceil {nm'}/{m} \right\rceil}^n (s_j - t_i)^r
$$
for $m' = 1,\ldots,m$. Letting $\tau_{(\mu,\nu)}(m') := \tau_{(\mu,\nu)}(d')$ be the tautological relations with respect to the basis $\{ t^\mu s^\nu \mid (\mu,\nu) \in \Delta \}$, we obtain with Theorem \ref{thm} that $A(M^\theta)$ is isomorphic to $A/\aa$, and $\aa$ is generated by
\begin{enumerate}
	\item $ax_1 + by_1$ and
	\item $\tau_{(\mu,\nu)}(m')$ for $1 \leq m' \leq m$ and $(\mu,\nu) \in \Delta$.
\end{enumerate}

In parts V and VI of Drezet's paper \cite{Drezet:88}, a recursive formula for the Betti numbers of some classes of Kronecker modules is given after showing (cf. \cite[Theorem 1]{Drezet:88}) that their cohomology with integral coefficients is torsion free and finitely generated. Ellingsrud and Str\o{}mme (cf. \cite[Theorem 6.9]{ES}) obtain a presentation of the Chow ring as the quotient of $A$ by the image $p(\cc)$, where $p: C \to A$ is the symmetrization map (as in the proof of Theorem \ref{thm}) and $\cc$ is the ideal generated by $l$ and all $wf^{d'}$. However, as $p$ is not multiplicative but just $A$-linear, $p(\cc)$ is not generated by the generators of $\cc$. They are hard to calculate by actually using $p$ and their number grows with a factor of $m!n!$.

Let us turn to some special cases.

\begin{ex} \label{ex_grass}
	Let $m=1$ and $r \geq n$. Then, the moduli space $M^\theta$ can easily be identified with the Grassmannian $\Gr_{r-n}(\kk^r)$. We choose $a = 1$ and $b = 0$ and obtain $x_1 \in \aa$. Therefore, $\AA := A(M^\theta)$ is isomorphic to $\Q[y_1,\ldots,y_n]/\aa'$ by sending $x_1$ to $0$. The ideal $\aa'$ is generated by the coefficients of $f^{(1)}(0,s_1,\ldots,s_n)$ in terms of the basis $s^\lambda$ with $\lambda_i \leq i-1$. We get
	$$
		f^{(1)}(0,s_1,\ldots,s_n) = s_n^r.
	$$
	Abbreviate $s := s_n$. We have $s^n = \sum_{i=1}^n (-1)^{i-1} y_is^{n-i}$. Letting $(\beta_j^{(\nu)} \mid \nu)$ be sequences such that $s^\nu = \sum_{i=1}^n \beta_i^{(\nu)} s^{n-i}$ for all $\nu \geq 0$, we obtain that these sequences fulfill the recursion
	$$
		\beta_j^{(\nu)} = \sum_{i=1}^n (-1)^{i-1}y_i\beta_j^{(\nu-i)}
	$$
	for $\nu \geq n$, and that have the initial values $\beta_j^{(\nu)} = \delta_{j,n-\nu}$ for $0 \leq \nu \leq n-1$. Using this terminology, the ideal $\aa'$ is given as
	$$
		\aa' = ( \beta_1^{(r)},\ldots,\beta_n^{(r)} ).
	$$
	We can determine these $\beta_j^{(r)}$ using Linear Algebra methods. Let $B$ be the $(n \times n)$-matrix
	$$
		B = \begin{pmatrix}
		    	0		& 1		&		& \\
					& \ddots	& \ddots	& \\
					&		& 0		& 1 \\
			(-1)^{n-1}y_n	& \ldots	& -y_2		& y_1 
		    \end{pmatrix}.
	$$
	Let $v_j^{(\nu)} := (\beta_j^{(\nu)},\ldots,\beta_j^{(\nu-1+n)})^T $. Then, $v_j^{(\nu)} = B^\nu v_j^{(0)} = B^\nu e_j$, and thus, $\beta_j^{(r)}$ is the last entry of the vector $B^{r-n}e_j$.\\
	The above description of the Chow ring $\AA$ is similar to one that is already known (cf. \cite[Example 14.6.6]{Fulton:98} and \cite[Theorem 1]{Groth:58:inter}): Let $c(t) = 1 + y_1t + \ldots +y_nt^n \in A[t]$. We can describe $\AA$ as the ring generated by $y_1,\ldots,y_n$ modulo the relations contained in the condition that the formal power series $c(t)^{-1} = \sum_i \delta_i t^i \in A[[t]]$ is actually a polynomial of degree at most $r-n$. This means $\AA$ is isomorphic to $\Q[y_1,\ldots,y_n]/(\delta_{r-n+1},\ldots,\delta_r)$. We get $\delta_0 = 1$ and for $\nu > 0$,
	$$
		\delta_\nu = (-1)^\nu \begin{vmatrix}
		                  	y_1    & y_2 &        & \ldots & y_\nu \\
		                  	1      & y_1 & y_2    & \ldots & y_{\nu-1} \\
		                  	0      & 1   & y_1    & \ddots & \vdots \\
		                  	\vdots &     & \ddots & \ddots & y_2 \\
		                  	0      & \ldots& 0    & 1      & y_1
		                  \end{vmatrix} =: (-1)^{\nu-1} d_\nu
	$$
	(defining $y_i := 0$ for $i > n$). Using Laplace's formula, we see that $d_\nu = \sum_{i=1}^n (-1)^{i-1} y_id_{\nu-i}$ for $\nu \geq n$. Letting $d_{-n+1} = \ldots = d_{-1} = 0$, this formula holds also true for every $\nu > 0$. Let $w^{(\nu)} := (d_{\nu-n+1},\ldots,d_\nu)^T$. We get that $d_{r-n+j}$ is the last entry of $B^{r-n}w^{(j)}$. As $w^{(0)},\ldots,w^{(n)}$ is also a basis of $\Q[y_1,\ldots,y_n]^n$, we have shown that
	$$
		( \beta_1^{(r)},\ldots,\beta_n^{(r)} ) = (d_{r-n+1},\ldots,d_r) = (\delta_{r-n+1},\ldots,\delta_r).
	$$
\end{ex}

\begin{ex} \label{ex_kronecker_23}
	Let $d = (2,3)$ and $r = 3$. Here, we have $f^{(1)} = (s_3 - t_1)^3(s_3 - t_2)^3$ and $f^{(2)} = (s_2 - t_1)^3(s_3 - t_1)^3$. A basis of $C = \Q[t_1,t_2, s_1, s_2, s_3]$ considered as an $A = \Q[x_1,x_2, y_1, y_2, y_3]$-module is given by the elements
	$$
		t_2^{\lambda_2} s_2^{\mu_2} s_3^{\mu_3},
	$$
	with $0 \leq \lambda_2,\mu_2 \leq 1$ and $0 \leq \mu_3 \leq 2$. Display the polynomials $f^{(1)}$ and $f^{(2)}$ as linear combinations of these basis vectors. We obtain presentations
	$$
		f^{(m')} = \sum_{0 \leq \lambda_2,\mu_2 \leq 1}\ \sum_{0 \leq \mu_3 \leq 2} \tau_{\lambda_2,\mu_2,\mu_3}(m') \cdot t_2^{\lambda_2}s_2^{\mu_2}s_3^{\mu_3}
	$$
	for $m' = 1,2$. We choose a linear relation, i.e. we choose integers $a$ and $b$ with $2a + 3b = 1$. For example, let $a = -1$ and $b = 1$. The linear relation thus obtained is $l = y_1 - x_1$. In order to simplify the tautological relations, we calculate in the ring $A/A\cdot l$, meaning we replace $x_1$ by $y_1$. Using \textsc{Singular} (you really do not want to do this by hand), we obtain, after a couple of simplifications, that $A(M^\theta)$ is isomorphic to $\Q[x_2,y_1,y_2,y_3]/\aa$, where $\aa$ is generated by
	\begin{center}
		\parbox{.45\textwidth}%
		{%
			\begin{itemize}
				\item $3x_2^2 - 3x_2y_2 + y_2^2 - y_1y_3$,
				\item $(3x_2 - 2y_2)y_3$,
				\item $x_2^3 - y_1y_2y_3 + y_3^2$,
				\item $-4x_2y_1 + y_1^3 + 3y_3$,
				\item $3x_2^2 - x_2y_1^2$,
			\end{itemize}
		} \hfill%
		\parbox{.45\textwidth}%
		{%
			\begin{itemize}
				\item $3x_2^2 + x_2y_2 - y_1^2y_2$,
				\item $x_2y_1y_2 - 3y_2y_3$,
				\item $3y_1^2 - 5y_2y_3$, and
				\item $x_2^3 - x_2y_1y_3$.
			\end{itemize}
		}%
	\end{center}
\end{ex}

	\bibliographystyle{abbrv}
	\bibliography{../../../../../sty-Files/Literature}
\end{document}